\documentclass[12pt,a4paper]{article}
\usepackage{latexsym, amsfonts, amsthm,amssymb,amscd,amsmath,makeidx}
\usepackage{enumerate}
\usepackage{exscale}
\usepackage{color} 
\usepackage{graphicx}
\usepackage{overpic} 
\usepackage{bm}
\usepackage{bbm}
\usepackage{fancyhdr}
\usepackage{mathrsfs} 

\definecolor{DarkGreen}{rgb}{0.2,0.6,0.2}

\def\green#1{\textcolor{DarkGreen}{#1}}
 \def\green#1{}

\def\om{\omega}

\numberwithin{equation}{section}

\def\Ind#1{{\mathbbmss 1}_{_{\scriptstyle #1}}}

\def\ua{\uparrow}
\def\da{\downarrow}

\def\wh{\widehat}

\def\ignore#1{}

\def\bR{{\mathbb R}}

\def\bZ{\mathbb Z}
\def\bN{\mathbb N}

\def\bT{\mathbb T}

\def\cX{{\mathscr X}}

\def\<{\langle}\def\>{\rangle}

\newtheorem{theorem}{Theorem}[section]
\newtheorem{proposition}[theorem]{Proposition}
\newtheorem{corollary}[theorem]{Corollary}
\newtheorem{lemma}[theorem]{Lemma}

\theoremstyle{definition}

\newtheorem{remark}[theorem]{Remark}

\normalsize

\usepackage{fancyhdr}
\title{\bf On a class of generalized Takagi functions\\ with linear pathwise quadratic variation}
\author{ \normalsize Alexander Schied\thanks{E-mail: {\tt schied@uni-mannheim.de}\hfill\break
The author gratefully acknowledges support by Deutsche Forschungsgemeinschaft through the Research Training Group RTG 1953}\\  \normalsize Department of Mathematics\\
 \normalsize University of Mannheim\\
 \normalsize 68131 Mannheim, Germany}
 
\date{\small First version: January 15, 2015\\
This version:  August 4, 2015\\
 }

\begin{document}

\maketitle
\vspace{-0.5cm}

\begin{abstract}We consider a class $\mathscr{X}$ of continuous functions on $[0,1]$ that is of interest from two different perspectives. First, it is closely related to sets of functions that have been studied as generalizations of the Takagi function. Second, each function in $\mathscr{X}$ admits  a linear pathwise quadratic variation and can thus serve as an integrator in F\"ollmer's pathwise It\=o calculus. We derive several uniform properties of the class $\mathscr{X}$. For instance, we compute the overall pointwise maximum, the uniform maximal oscillation, and the exact uniform modulus of continuity  for all functions in $\mathscr{X}$. 
Furthermore, we give an example of a pair $x,y\in\mathscr{X}$ for which the quadratic variation of the sum $x+y$ does  not exist. \end{abstract}


\bigskip


\section{Introduction}

In this note, we study a class $\cX$ of continuous functions on $[0,1]$ that is of interest from several different perspectives. On the one hand, just as typical Brownian sample paths, each function $x\in\cX$ admits the linear pathwise quadratic variation, $\<x\>_t=t$, in the sense of F\"ollmer~\cite{FoellmerIto} and therefore can serve as an integrator in F\"ollmer's pathwise It\=o calculus. On the other hand, $\cX$ is a subset, or  has a nonempty intersection, with classes of functions that have been studied as generalizations  of Takagi's celebrated example~\cite{Takagi} of a nowhere differentiable continuous function. We will now explain the connections of our results with these two separate strands of literature.

\subsection{Contributions to F\"ollmer's pathwise It\=o calculus}
In 1981,  F\"ollmer~\cite{FoellmerIto} proposed a pathwise 
version of It\=o's formula, which, as a consequence, yields a strictly pathwise definition of the It\=o integral as a limit of Riemann sums. 
Some recent developments have led to a renewed interest in this pathwise approach. Among these is the conception of  \emph{functional} pathwise It\=o calculus by Dupire~\cite{Dupire} and Cont and Fourni\'e~\cite{ContFournieJFA, ContFournieAOP}, which for instance is crucial in defining partial differential equations on   path space~\cite{EkrenKellerTouzi}. Another source for the renewed interest in pathwise It\=o calculus stems from the growing awareness of model ambiguity in mathematical finance and the resulting desire to reduce the reliance on probabilistic models; see, e.g.,~\cite{FoellmerSchiedBernoulli} for a recent survey and~\cite{Benderetal1,BickWillinger,DavisRavalObloij,FoellmerECM,SchiedCPPI,SchiedStadje} for case studies with successful applications of pathwise It\=o calculus to financial problems. A systematic introduction to pathwise It\=o calculus, including an English translation of~\cite{FoellmerIto}, is provided in~\cite{Sondermann}.

A function $x\in C[0,1]$ can serve as an integrator in F\"ollmer's pathwise It\=o calculus if it admits a continuous  pathwise quadratic variation $t\mapsto \<x\>_t$ along a given refining sequence  of partitions of $[0,1]$. This condition is satisfied whenever $x$ is a  sample path of a continuous semimartingale, such as Brownian motion,  and does not belong to a certain nullset. This nullset, however,   is generally not known explicitly, and so it is not possible to tell whether a specific realization $x$ of Brownian motion does indeed admit a continuous pathwise  quadratic variation. The first purpose of this note is to provide a  rich class $\cX$ of continuous functions  that can be constructed in a straightforward manner and that do admit the  nontrivial  pathwise quadratic variation $\<x\>_t=t$ for all $x\in\cX$. The functions in $\cX$ can  thus  be used as a class of test integrators in pathwise It\=o calculus. Our corresponding result, Proposition~\ref{Takagi quad var prop}, slightly extends a previous result by Gantert~\cite{GantertDiss,Gantert}, from which it follows that $\<x\>_1=1$ for all $x\in\cX$.

 Still within this context, a second purpose of this note is to investigate whether the existence of $\<x\>$ and $\<y\>$ implies the existence of $\<x+y\>$ (or, equivalently,  the existence of the pathwise quadratic covariation $\<x,y\>$).  For typical sample paths of a continuous semimartingale, this implication is always true, but the corresponding nullset will depend on both $x$ and $y$. In the literature on pathwise It\=o calculus, however, it has been taken for granted  that  the existence of $\<x+y\>$  cannot be deduced from the existence of $\<x\>$ and $\<y\>$. 
In Proposition~\ref{MountTakagi_var_prop} we will now give an example of two  functions $x,y\in\cX$ for which $\<x+y\>$  does indeed not exist.  To the knowledge of the author, such an example has so far been missing from the literature. 

\subsection{Contributions to the theory of generalized Takagi functions}

In 1903, Takagi~\cite{Takagi} proposed an example of a continuous function on $[0,1]$ that is nowhere differentiable. This function has since been rediscovered several times and its properties have been studied extensively; see the recent surveys by Allaart and Kawamura~\cite{AllaartKawamura} and Lagarias~\cite{Lagarias}. While the original Takagi function itself does not belong to our class $\cX$, there are at least two classes of functions whose study was motivated by the Takagi function and that are intimately connected with $\cX$. One  family of functions is the \lq\lq Takagi class"  introduced in 1984 by Hata and Yamaguti~\cite{HataYamaguti}. Similar but more restrictive function classes were introduced earlier by Faber~\cite{Faber} or Kahane~\cite{Kahane}. The Takagi class  has a nonempty intersection with $\cX$  but neither one is included in the other. More recently, Allaart~\cite{Allaartflexible} extended the Takagi class to a more flexible class of functions. This family now contains $\cX$. By extending arguments given by K\^ono~\cite{Kono} for the Takagi class, Allaart~\cite{Allaartflexible} studies in particular the moduli of continuity of certain functions in his class. 

In contrast to these previous studies, the focus of this paper is not so much on the individual features of functions $x\in\cX$ but rather on uniform properties of the entire class $\cX$. Here we compute the overall pointwise maximum, the uniform maximal oscillation, and the exact uniform modulus of continuity  for all functions in $\cX$. In these computations, we cannot use previous methods that were conceived for the analysis of the Takagi functions and its generalizations. For instance, neither  the result and arguments from K\^ono~\cite{Kono}  nor the ones from Allaart~\cite{Allaartflexible} apply to the modulus of continuity of functions in $\cX$, and a suitable extension of the previous approaches must be developed. This new extension exploits the self-similar structure of $\cX$ and its members.

A special role in our analysis will be played by  the function $\wh x$, defined in \eqref{Takagi-style fct} below. It has previously appeared in the work of  Ledrappier~\cite{Ledrappier}, who studied the Hausdorff dimension of its graph, and in Gantert~\cite{GantertDiss,Gantert}. Here we will determine its global maximum and its exact modulus of continuity. In particular the results on the global maximum of $\wh x$ will be needed in our analysis of the uniform properties of $\cX$, but these results are also interesting in their own right.

\bigskip

This paper is organized as follows. In the subsequent Section~\ref{ResultsSection} we first introduce our class $\cX$ and then discuss its uniform properties in Theorems~\ref{max osc thm} and~\ref{modulus continuity thm} and Corollary~\ref{compact cor}. We then recall F\"ollmer's~\cite{FoellmerIto} notions of pathwise quadratic variation and covariation and state our corresponding results.  All proofs are given in Section~\ref{ProofsSection}. 

\section{Statement of results}\label{ResultsSection}

Recall that the \emph{Faber--Schauder functions}  are defined as
$$e_\emptyset(t):=t,\qquad e_{0,0}(t):=(\min\{t,1-t\})^+,\qquad e_{m,k}(t):=2^{-m/2}e_{0,0}(2^m t-k)
$$
for $t\in\bR$, $m=1,2,\dots$, and $k\in\bZ$. 
The graph of $e_{m,k}$  looks like a wedge with height $2^{-\frac {m+2}2}$,  width $2^{-m}$, and center at $t=(k+\frac12)2^{-m}$. In particular,  the functions $e_{m,k}$ have disjoint support   for  distinct $k$ and fixed $m$. Now let coefficients $\theta_{m,k}\in\{-1,+1\}$  be given and define for $n\in\bN$ the continuous functions
\begin{equation}\label{Mount Takagi approx xN}
x^n(t):=\sum_{m=0}^{n-1}\sum_{k=0}^{2^m-1}\theta_{m,k}e_{m,k}(t),\qquad 0\le t\le 1.
\end{equation}
It is well known (see, e.g.,~\cite{Allaartflexible}) and easy to see that, due the uniform boundedness of the coefficients $\theta_{m,k}$,  the functions $x^n(t)$ converge uniformly in $t$ to a continuous function $x(t)$ as $n\ua\infty$. 
Let us denote by 
$$\cX:=\Big\{x\in C[0,1]\,\Big|\,x= \sum_{m=0}^{\infty}\sum_{k=0}^{2^m-1}\theta_{m,k}e_{m,k}\text{ for  coefficients }\theta_{m,k}\in\{-1,+1\}\Big\}$$
the class of limiting functions arising in this way. A function $x\in \cX$ belongs to the \lq\lq Takagi class" introduced by Hata and Yamaguti~\cite{HataYamaguti} if and only if the coefficients $\theta_{m,k}$ in \eqref{Mount Takagi approx xN}
 are independent of $k$. Moreover, $\cX$ is a subset of the more flexible class of generalized Takagi functions studied by Allaart~\cite{Allaartflexible}. The original Takagi function, however, is obtained by taking $\theta_{m,k}=2^{-m/2}$ and therefore does not belong to $\cX$. 
 
\begin{remark}[\bfseries On similarities with Brownian sample paths]\label{Brownian remark} The functions in $\cX$ can exhibit   interesting fractal structures; see  Figure~\ref{Mount_TakagiFig}. Figure~\ref{Mount_TakagirandFig}, on the other hand, displays some similarities with the sample paths of a Brownian bridge. This similarity is not surprising since the well-known  L\'evy--Ciesielski construction of the Brownian bridge consists in replacing the coefficients $\theta_{m,k}\in\{-1,+1\}$ with independent standard normal random variables (see, e.g.,~\cite{KaratzasShreve}). As a matter of fact, using arguments of de Rham~\cite{deRham} and Billingsley~\cite{Billingsley1982}, it was shown in~\cite[Theorem 3.1 (iii)]{Allaartflexible} that functions in $\cX$ share with Brownian sample paths the property of being   nowhere differentiable. Moreover, Ledrappier~\cite{Ledrappier} showed that the Hausdorff dimension of the graph of the function  
 \begin{align}\label{Takagi-style fct}
 \wh x:=\sum_{m=0}^{\infty}\sum_{k=0}^{2^m-1}e_{m,k}
\end{align}
is the same as that of the graphs of typical Brownian trajectories, namely $3/2$. In Proposition~\ref{Takagi quad var prop} we will see, moreover, that the functions in $\cX$ have the same pathwise quadratic variation as Brownian sample paths. 
\end{remark} 

Our first result is concerned with (uniform) maxima and oscillations of the functions in $\cX$. It will also be concerned with the function $\wh x$ defined in \eqref{Takagi-style fct}, a  function that  will play a special role throughout our analysis. The maximum of the original Takagi function was computed by Kahane~\cite{Kahane}, but his method does not apply in our case, and more complex arguments are needed here.

 \begin{figure}[h]
\centering
\begin{minipage}[b]{8cm}
\begin{overpic}[width=8cm]{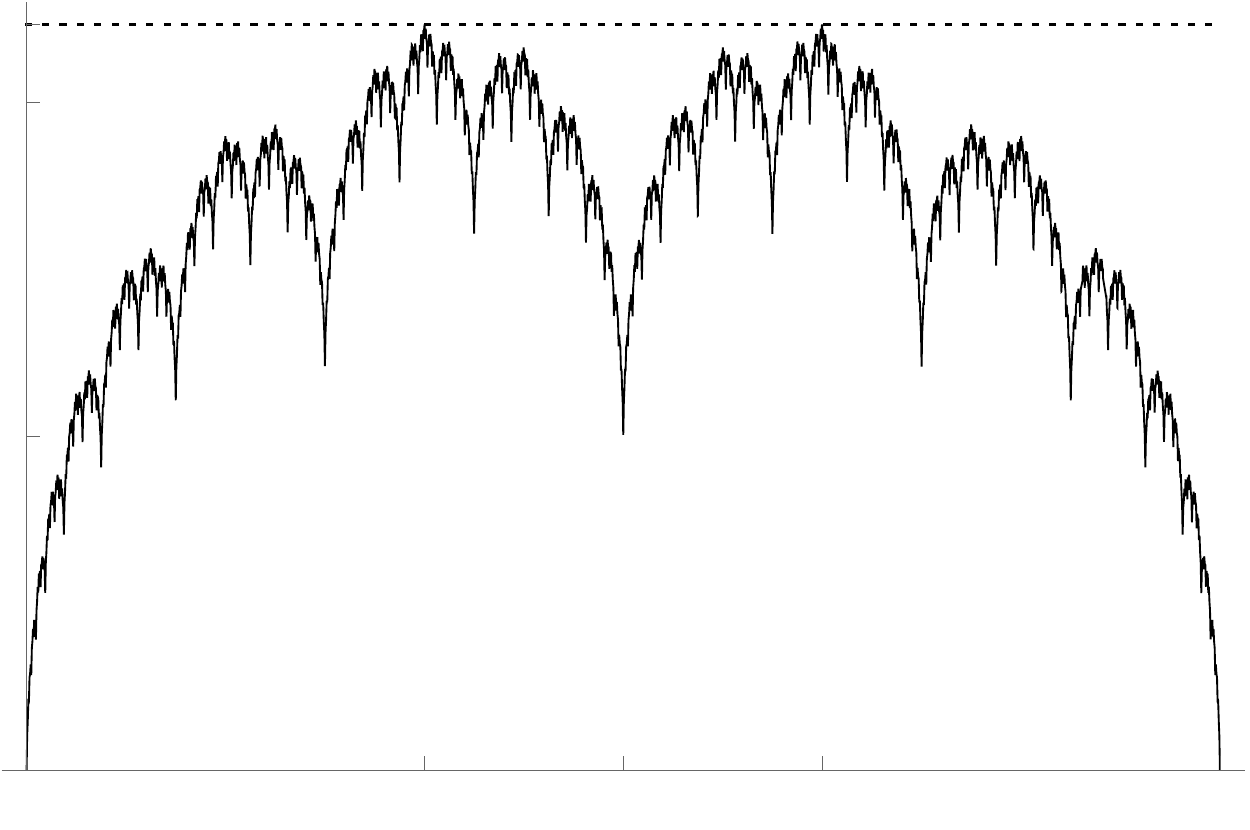}
\put(31,0){\scriptsize{1/3}}
\put(47,0){\scriptsize{1/2}}
\put(63,0){\scriptsize{2/3}}
\put(96,0){\scriptsize{1}}
\put(-5,31){\scriptsize{1/2}}
\put(-1,57){\scriptsize{1}}
\put(-16,64){\scriptsize{$\frac13(2+\sqrt2)$}}
\end{overpic}\\
\end{minipage}\qquad
\begin{minipage}[b]{8cm}
\includegraphics[width=8cm]{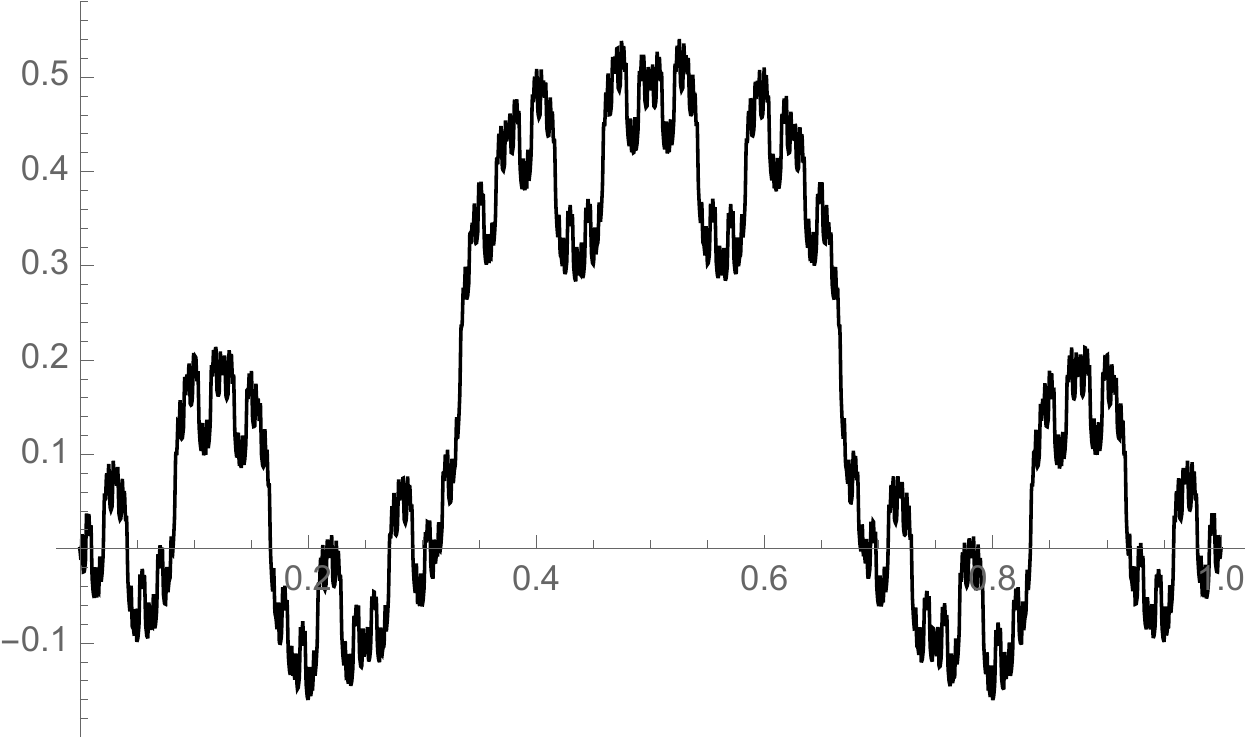}\\
\end{minipage}
\begin{minipage}[b]{8cm}
\includegraphics[width=8cm]{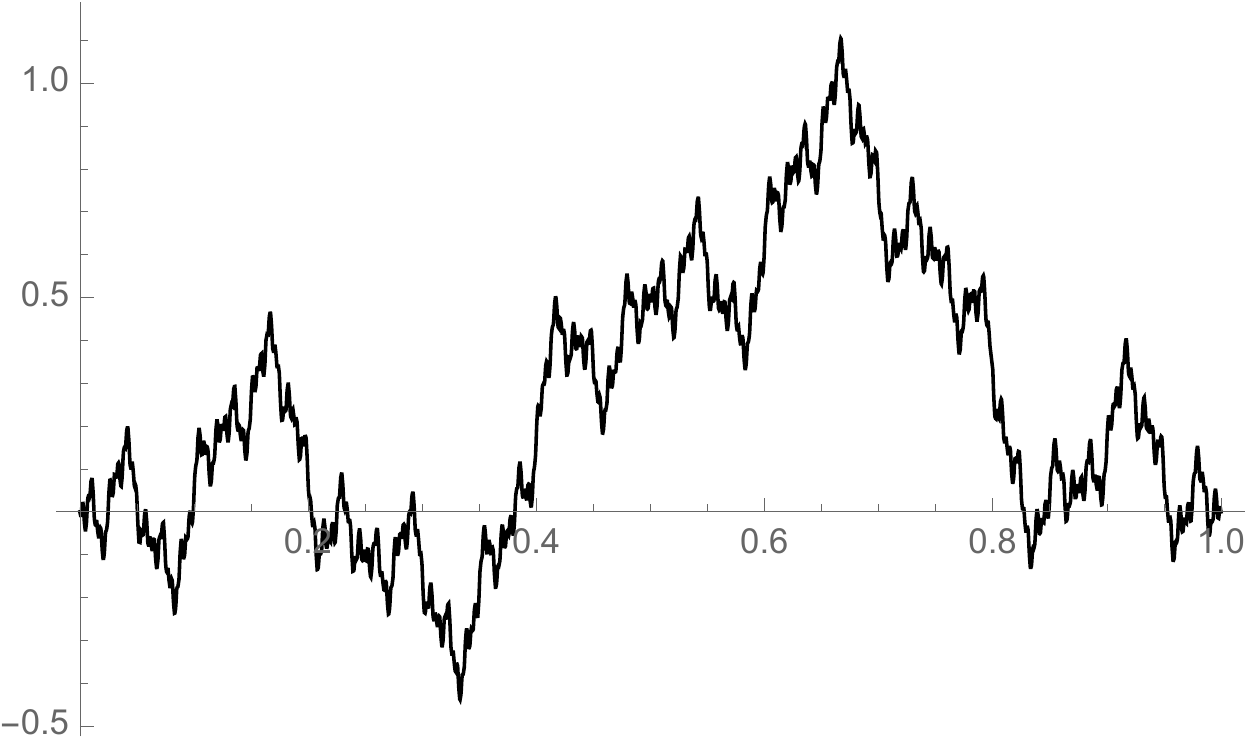}\\
\end{minipage}\hskip1.0cm
\begin{minipage}[b]{8cm}
\includegraphics[width=8
cm]{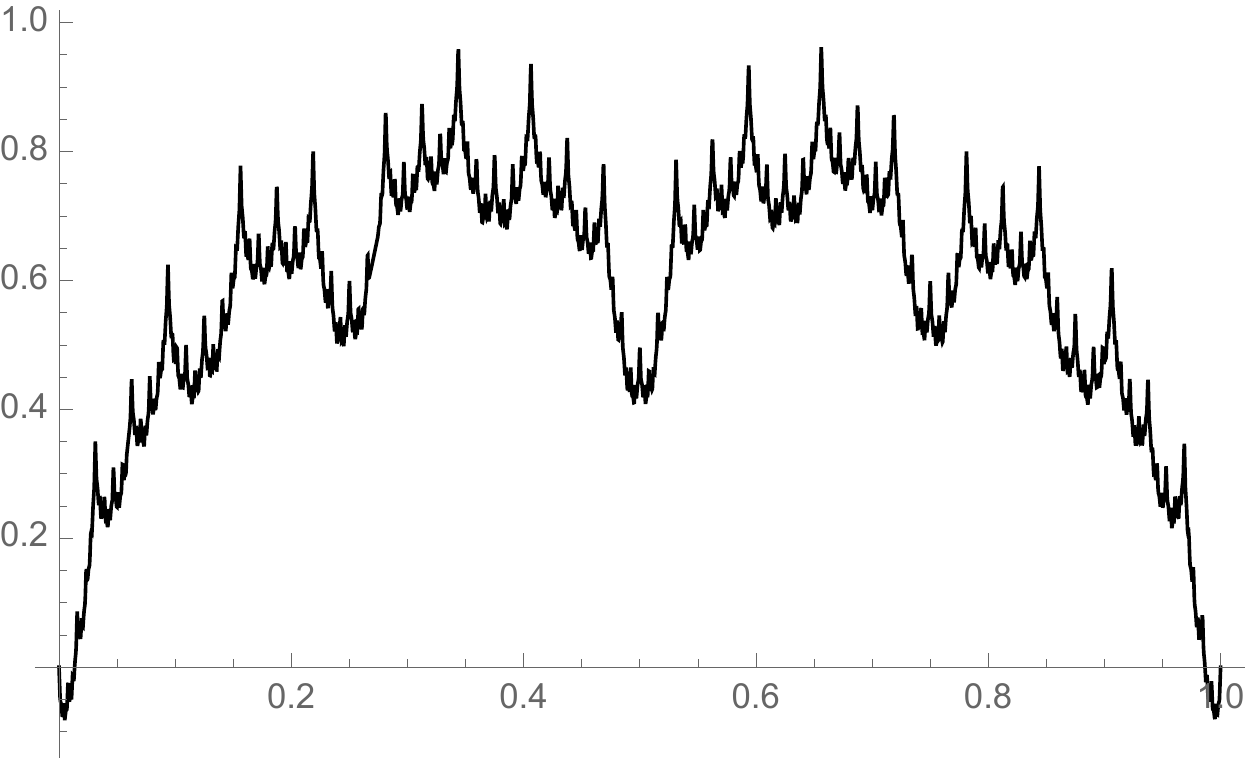}\\
\end{minipage}
 \caption{Plots of functions in $\cX$  for various choices of $\theta_{m,k}$.  The upper left-hand panel shows the function $\wh x$, which is defined through $\theta_{m,k}=1$, together with its global maximum. The upper right-hand panel corresponds to $\theta_{m,k}=(-1)^m$, the lower left-hand panel to $\theta_{m,k}=(-1)^{m+k}$, and the lower right-hand panel to $\theta_{m,k}=(-1)^{\lfloor m/5\rfloor}$. }
\label{Mount_TakagiFig}
\end{figure}
\begin{figure}[htbp]
\centering
\begin{minipage}[b]{8cm}
\includegraphics[width=8cm]{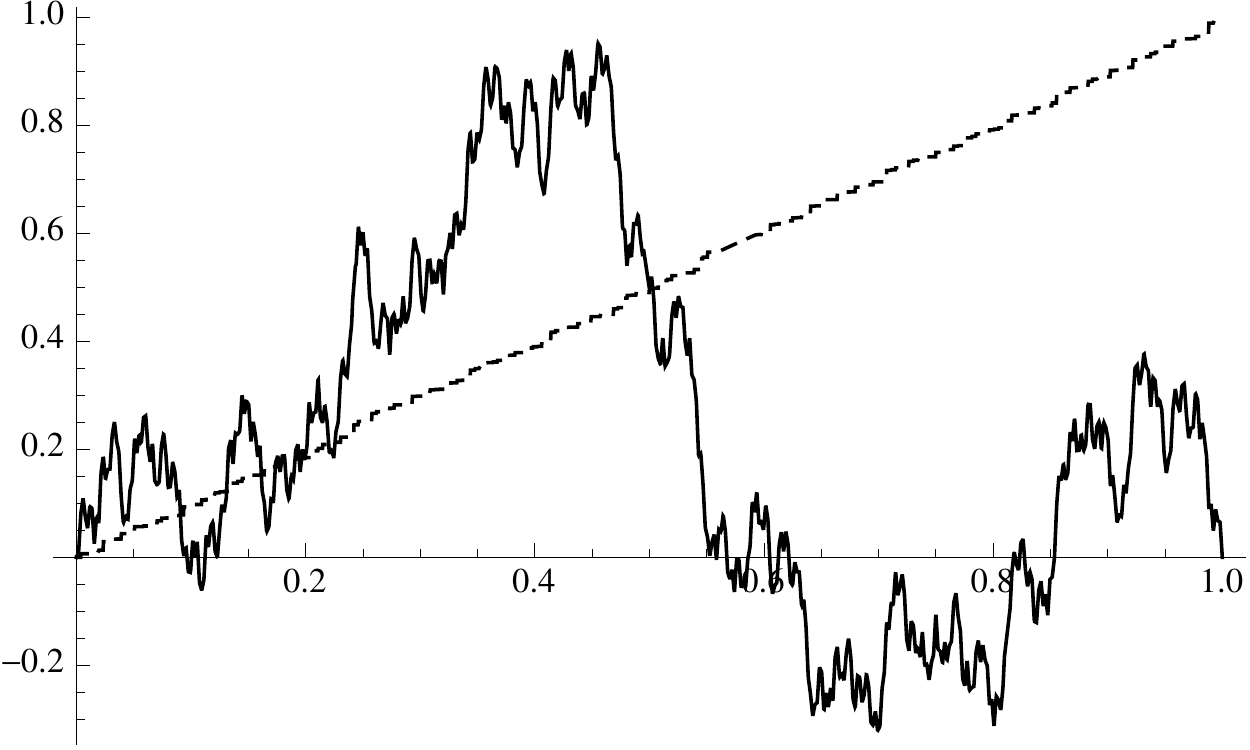}\\
\end{minipage}\qquad
\begin{minipage}[b]{8cm}
\includegraphics[width=8cm]{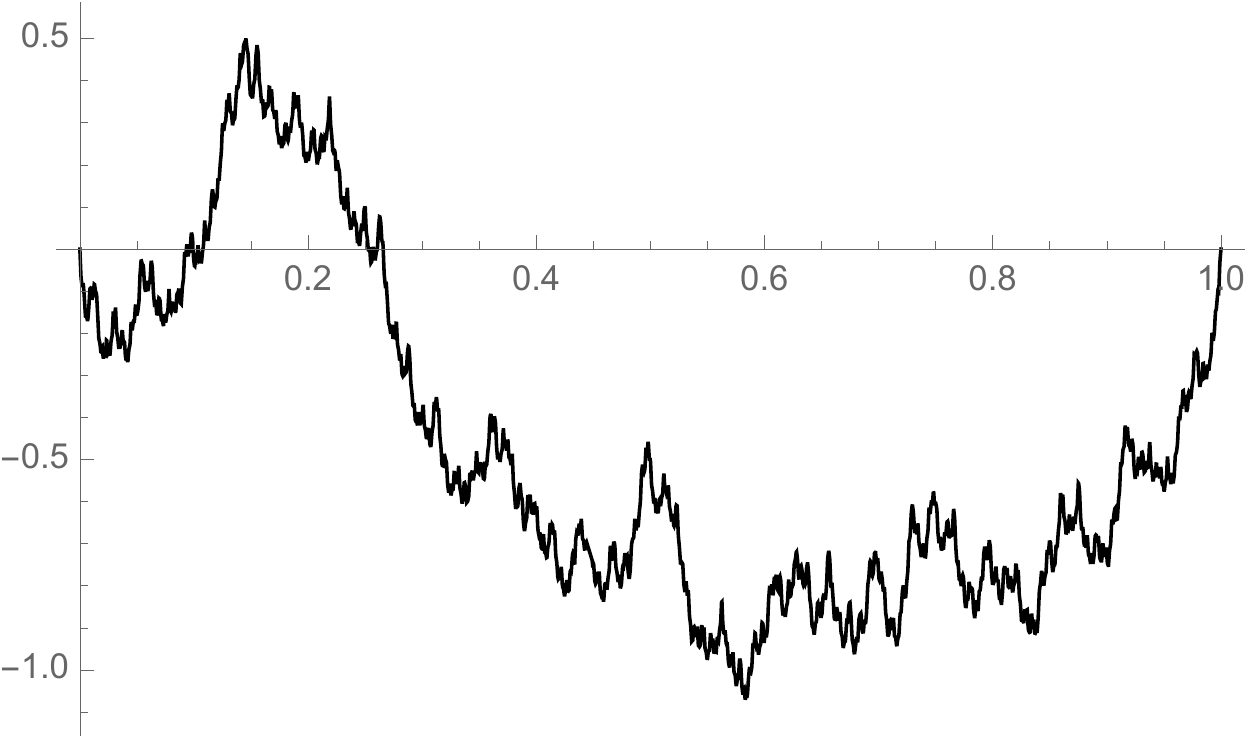}\\
\end{minipage}
\caption{Plots of  $x\in\cX$  when the coefficients $\theta_{m,k}$ form an independent and identically distributed $\{-1,+1\}$-valued random sequence such that  $\theta_{m,k}=+1$ with  probability $\frac12$ (left) and $\frac14$ (right). The dashed line corresponds to  the approximation $\<x\>^8$ of the quadratic variation along the $8^{\text{th}}$ dyadic partition $\bT_8$. }\label{Mount_TakagirandFig}
\end{figure}

\begin{theorem}[\bfseries Uniform maximum and oscillations]\label{max osc thm} The class $\cX$ has the following uniform properties.
\begin{enumerate}
\item  The uniform maximum of functions in $\cX$ is attained by $\wh x$ and  given by
$$\max_{x\in\cX}\max_{t\in[0,1]}|x(t)|=\max_{t\in[0,1]}\wh x(t)=\frac13(2+\sqrt2).$$
Moreover, the maximum of $\wh x(t)$ is attained at  $t=\frac13$ and $t=\frac23$.

\item The maximal  uniform  oscillation of functions in $\cX$ is 
$$\max_{x\in\cX}\max_{s,t\in[0,1]}|x(t)-x(s)|=\frac16(5+4{\sqrt2}),$$
where the respective maxima are attained at $s=1/3$,  $t=5/6$, and \begin{align}\label{xstar def eq}
x^*:=e_{0,0}+\sum_{m=1}^\infty\bigg(\sum_{k=0}^{2^{m-1}-1}e_{m,k}-\sum_{\ell=2^{m-1}}^{2^m-1}e_{m,\ell}\bigg);
\end{align}
\end{enumerate}
\end{theorem}

We refer to  Figure \ref{xstar fig} for a plot of the function $x^*$.

In our next result, we will investigate the modulus of continuity of $\wh x$ and the uniform modulus of continuity of  the class $\cX$. K\^ono~\cite{Kono}  analyzed the moduli of continuity for some functions in the Takagi class of Hata and Yamaguti~\cite{HataYamaguti}, and Allaart~\cite{Allaartflexible} later extended this result. However, neither the result and arguments from~\cite{Kono}  nor the ones from~\cite{Allaartflexible} apply to the functions in $\cX$,  because the sequence $a_m:=2^{m/2}$ is not bounded.  To state our results,  let us denote by 
\begin{align}\label{mu eq}
\nu(h):=\lfloor-\log_2h\rfloor,\qquad h>0,
\end{align}
 the integer part of $\log_2\frac1h$, and define
 \begin{align*}
 \om (h)&:=\Big(1+\frac1{\sqrt2}\Big)h2^{\nu(h)/2}+\frac13(\sqrt8+2)2^{-\nu(h)/2}. \end{align*}
Note that $\om (h)$ is of the order  $O(\sqrt h)$ as $h\da0$. More precisely, 
\begin{align*}
\liminf_{h\da0}\frac{\om (h)}{\sqrt h}=2\sqrt{\frac43+\sqrt2},\qquad\qquad&\limsup_{h\da0}\frac{\om (h)}{\sqrt h}=\frac16(11+7\sqrt2).\end{align*}
These exact limits will however not be needed in the remainder of the paper.

\medskip

\begin{theorem}[\bfseries Moduli of continuity]\label{modulus continuity thm}$ $
\begin{enumerate}
\item The function $\wh x$ has $\om $ as its modulus of continuity. More precisely,
$$\limsup_{h\da0}\max_{0\le t\le 1-h}\frac{|\wh x(t+h)-\wh x(t)|}{\om (h)}=1.
$$
\item An exact uniform modulus of continuity for functions in $\cX$ is given by $\sqrt2\om$. That is, 
 $$\limsup_{h\da0}\sup_{x\in\cX}\max_{0\le t\le 1-h}\frac{|  x(t+h)-  x(t)|}{ \om (h)}=\sqrt2.$$
 Moreover, the above supremum over functions $x\in\cX$ is attained by the function $x^*$ defined in \eqref{xstar def eq} in the sense that 
   $$\limsup_{h\da0}\max_{0\le t\le 1-h}\frac{|  x^*(t+h)-  x^*(t)|} {\om (h)}=\sqrt2.$$
\end{enumerate}
\end{theorem}

\medskip

\begin{remark}In the proof of Theorem \ref{modulus continuity thm}, we will actually show the following upper bounds that are stronger than the corresponding statements in the theorem:
$$|\wh x(t+h)-\wh x(t)|\le \om (h)\qquad\text{and}\qquad \sup_{x\in\cX}| x(t+h)-x(t)|\le\sqrt2 \om (h)
$$
for all $h\in[0,1)$ and $t\in [0,1-h]$.
\end{remark}

\medskip

\begin{corollary}\label{compact cor}$\cX$ is  a compact subset of $C[0,1]$ with respect to the topology of uniform convergence.
\end{corollary}

\medskip

Theorem~\ref{modulus continuity thm} (b) implies moreover that each $x\in\cX$ is H\"older continuous with exponent $\frac12$ and hence admits a finite 2-variation in the sense that\begin{equation}\label{finite 2-variation}
\sup_{\bT}\sum_{t\in\bT}(x(t')-x(t))^2<\infty,
\end{equation}
where the supremum is taken over all partitions $\bT$ of $[0,1]$ and $t'$ denotes the successor of $t$ in $\bT$, i.e., 
$$t'=\begin{cases}\min\{u\in\bT\,|\,u>t\}&\text{if $t<1$,}\\
1&\text{if $t=1$.}
\end{cases}
$$
Each $x\in\cX$ can therefore serve as an integrator in the pathwise integration theory of rough paths; see, e.g., Friz and Hairer~\cite{FrizHairer}. A different pathwise integration theory was proposed earlier by F\"ollmer~\cite{FoellmerIto}. It is based on the following notion of \emph{pathwise quadratic variation}. Instead of considering the supremum over \emph{all} partitions as in \eqref{finite 2-variation}, one fixes an increasing sequence of partitions $\bT_1\subset\bT_2\subset\cdots$ of $[0,1]$ such that the mesh of $\bT_n$ tends to zero; such a sequence $(\bT_n)_{n\in\bN}$ will be called a \emph{refining sequence of partitions}.  For  $x\in C[0,1]$ one then defines the sequence
\begin{equation}
  \label{quadVarApprox}
 \<
x\>_{t}^n:= \sum_{s\in\bT_n,\, s\le  t} 
(x(s')-x(s))^2 .
\end{equation}
The function $x\in C[0,1]$
is  said to \emph{admit the continuous quadratic variation $\<x\>$ along the sequence $(\bT_n)$} if for all $t\in[0,1]$ the limit
\begin{align}\label{quadVarLimit}
\<x\>_{t}:=\lim_{n\ua\infty}\<x\>^n_{t}
\end{align}
exists, and if $t\mapsto\<x\>_{t}$ is a continuous function. F\"ollmer's pathwise It\=o calculus uses this\footnote{Pathwise It\=o calculus also works for c\`adl\`ag  functions $x$, but this requires that the continuous part of $x$ admits a continuous quadratic variation along $(\bT_n)$; see~\cite{FoellmerIto} and~\cite{ContFournieJFA}. For this reason we will concentrate here on the case of continuous functions $x$.} class of functions $x$ as integrators. 

For given $x\in C[0,1]$, the approximations $\<x\>^n_{t}$ are typically not monotone in $n$, and so it is not clear \emph{a priori} whether the limit in \eqref{quadVarLimit} exists. Moreover, even if the limit exists, it may depend strongly on the particular choice of the underlying sequence of partitions. For instance, it is known that for any $x\in C[0,1]$ there exists a refining sequence $(\bT_n)_{n\in\bN}$ of partitions such that the quadratic variation of $x$ along $(\bT_n)_{n\in\bN}$ vanishes identically; see~\cite[p. 47]{Freedman}. It is also not difficult to construct $x\in C[0,1]$ for which the limit in \eqref{quadVarLimit} exists but satisfies $\<x\>_{t}=\Ind{]1/2,1]}(t)$ and is hence discontinuous. On the other hand, it is easy to see that the quadratic variation of a continuous function with bounded variation exists and vanishes along every refining sequence of partitions. Functions that do admit a nontrivial quadratic variation for some refining sequence of partitions must hence be of infinite total variation.

So the first question that arises in this context is how one can obtain functions that do admit a continuous quadratic variation along a given refining  sequence of partitions, $(\bT_n)_{n\in\bN}$? Of course one can take  all sample paths of a Brownian motion (or, more generally, of a continuous semimartingale) that are not contained in a certain null set $A$. But $A$ is generally not given explicitly, and so it is not possible to tell whether a specific realization $x$ of Brownian motion does indeed admit the quadratic variation $\<x\>_{t}=t$ along $(\bT_n)_{n\in\bN}$. Moreover, this selection principle for functions $x$ lets a probabilistic model enter through the backdoor, although the initial purpose of pathwise It\=o calculus was to get rid of probabilistic models altogether.

In the following proposition,  we show that each $x\in\cX$ admits the linear pathwise quadratic variation $\<x\>_t=t$ for $t\in[0,1]$ along the sequence of dyadic partitions: 
\begin{equation}\label{dyadic partitions}
\bT_n:=\{k2^{-n}\,|\,k=0,\dots, 2^n\},\qquad n=1,2,\dots
\end{equation}
 This slightly extends a result by  Gantert~\cite{GantertDiss,Gantert}, from which it follows that $\<x\>_1=1$ for all $x\in\cX$. Our proposition implies that each $x\in\cX$ can serve as an integrator in F\"ollmer's pathwise It\=o calculus.

\medskip

\begin{proposition}\label{Takagi quad var prop}Every $x\in\cX$ admits the quadratic variation $\langle x\>_{t}=t$ along the sequence $(\bT_n)$ from \eqref{dyadic partitions}.\end{proposition}

 \medskip
 
The second question that we will address in this context is concerned with a standard assumption that is made in pathwise It\=o calculus whenever the covariation of two functions $x,y\in C[0,1]$ is needed. Let
\begin{align}\label{}
\<x,y\>_{t}^n:=\sum_{s\in\bT_n,\, s\le t}(x(s')-x(s))(y(s')-y(s))
\end{align}
and observe that 
\begin{align}\label{polarization in approx}
\<x,y\>_{t}^n=\frac12\Big(\<x+y\>_{t}^n-\<x\>_{t}^n-\<y\>_{t}^n\Big).
\end{align}
If $x$ and $y$ admit the continuous quadratic variations $\<x\>$ and $\<y\>$ along $(\bT_n)_{n\in\bN}$, then it follows from \eqref{polarization in approx} that 
the \emph{covariation of $x$ and $y$},
\begin{align}\label{covariation limit}
\<x,y\>_{t}:=\lim_{n\ua\infty} \<x,y\>_{t}^n,
\end{align}
exists along $(\bT_n)_{n\in\bN}$ and is continuous in $t$ if and only if $x+y$ admits a  continuous quadratic variation  along $(\bT_n)_{n\in\bN}$. When $x$ and $y$ are  sample paths of Brownian motion or, more generally, of a continuous semimartingale, the quadratic variation $\<x+y\>$, and hence $\<x,y\>$, will always exist almost surely. 
But for arbitrary functions $x,y\in C[0,1]$ it has so far not been possible to reduce the existence of the limit in \eqref{covariation limit} to the existence of $\<x\>$ and $\<y\>$.  In our next proposition we will provide an example of two functions $x,y\in\cX$ for which the limit in \eqref{covariation limit}  does not exist, even though $\<x\>_{t}=t=\<y\>_{t}$. This shows that the existence of $\<x,y\>$ and $\<x+y\>$ is not implied by the existence of $\<x\>$ and $\<y\>$. It follows in particular that the class of functions that admit a continuous quadratic variation along $(\bT_n)_{n\in\bN}$ is not a vector space. To the knowledge of the author, a corresponding example has so far been missing from the literature.

 \medskip
 
 \begin{proposition}\label{MountTakagi_var_prop}Consider the sequence $(\bT_n)_{n\in\bN}$ of dyadic partitions \eqref{dyadic partitions} and the functions $$\wh x=\sum_{m=0}^\infty\sum_{k=0}^{2^m-1}e_{m,k}\qquad\text{and}\qquad y=\sum_{m=0}^\infty\sum_{k=0}^{2^m-1}(-1)^me_{m,k},
$$
which belong to $\cX$ and hence admit the quadratic variation $\langle \wh x\>_{t}=t=\<y\>_{t}$ along $(\bT_n)$. Then
\begin{align}\label{MountTakagi_var_prop eq1}
\lim_{n\ua\infty}\<\wh x+y\>^{2n}_t=\frac43t\qquad\text{and}\qquad\lim_{n\ua\infty}\<\wh x+y\>^{2n+1}_t=\frac83t
\end{align}
and
\begin{align}\label{MountTakagi_var_prop eq2}
\lim_{n\ua\infty}\<\wh x,y\>^{2n}_t=-\frac13t\qquad\text{and}\qquad\lim_{n\ua\infty}\<\wh x,y\>^{2n+1}_t=\frac13t.
\end{align}
In particular, for $t>0$, the limits of $\<\wh x+y\>^n_t$ and $\<\wh x,y\>_t^n$ do not exist as $n\ua\infty$, but $\wh x+y$ admits different continuous quadratic variations along the two refining sequences $(\bT_{2n})_{n\in\bN}$ and $(\bT_{2n+1})_{n\in\bN}$.
 \end{proposition}
 
 \medskip
 
See the two upper panels in Figure~\ref{Mount_TakagiFig} for plots of the two functions $\wh x$ and $y$ occurring in Proposition 
\ref{MountTakagi_var_prop}. See Figure~\ref{Mount_Takagi_moQV Fig} for a plot of 
\begin{equation}\label{Mount_Takagi_moQV eq}
\wh x+y=\sum_{m=0}^\infty\sum_{k=0}^{2^{2m}-1}2e_{2m,k}.
\end{equation}

\begin{figure}[htbp]
\begin{center}
\includegraphics[width=10cm]{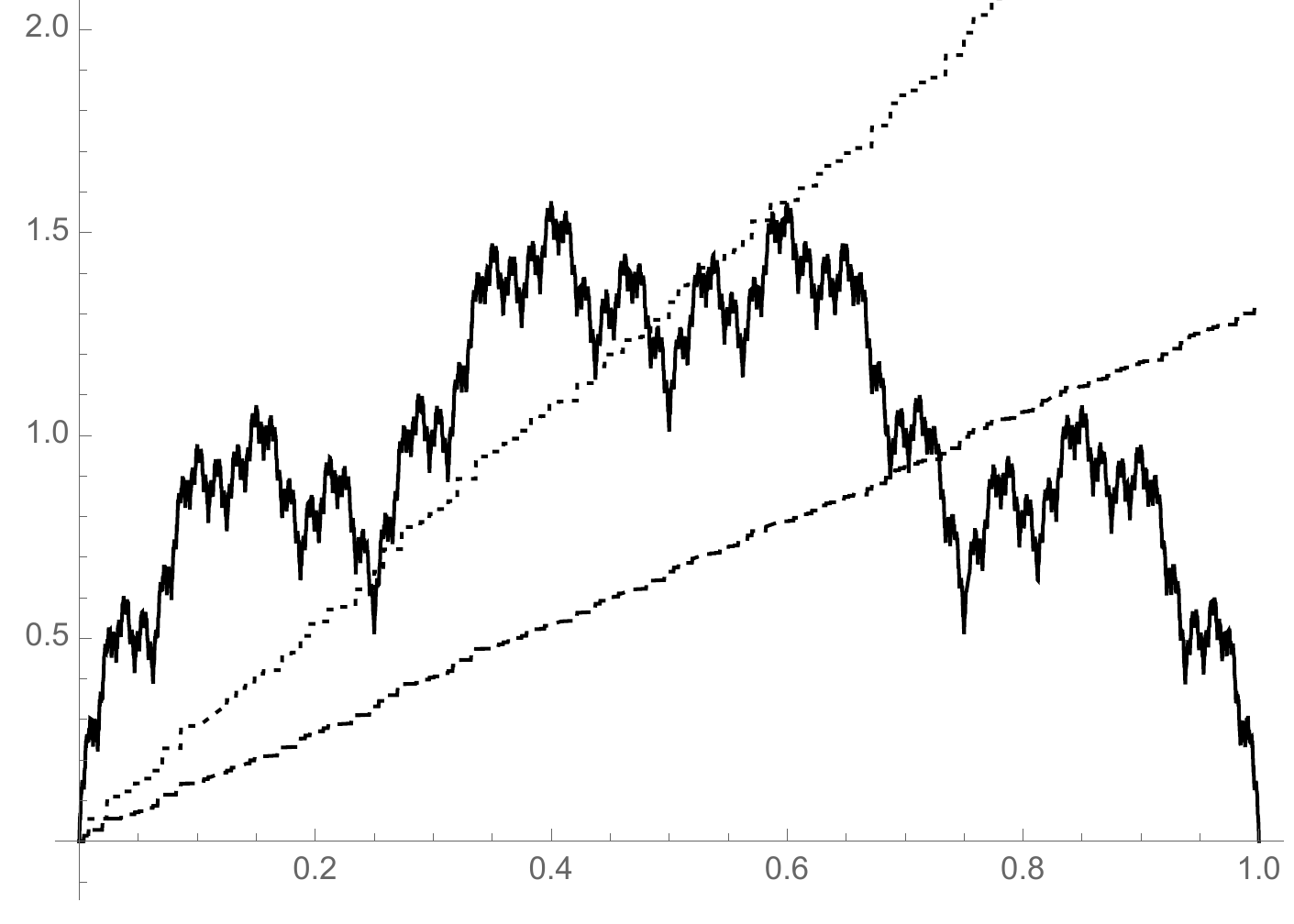}
\caption{Plot of the function $\wh x+y$ defined in \eqref{Mount_Takagi_moQV eq}. The dotted line is $\<\wh x+y\>^7$, the dashed line is $\<\wh x+y\>^8$.}\label{Mount_Takagi_moQV Fig}
\end{center}
\end{figure}

\section{Proofs}\label{ProofsSection}

\subsection{Proof of Theorem~\ref{max osc thm} }

We start with the following lemma, which computes  maxima and maximizers of the functions 
$$\wh x^n(t):=\sum_{m=0}^{n-1}\sum_{k=0}^{2^m-1}e_{m,k}(t),\qquad t\in[0,1]\text{ and }n=1,2,\dots;
$$
see Figure \ref{Jacobsthal Fig}
 for an illustration.
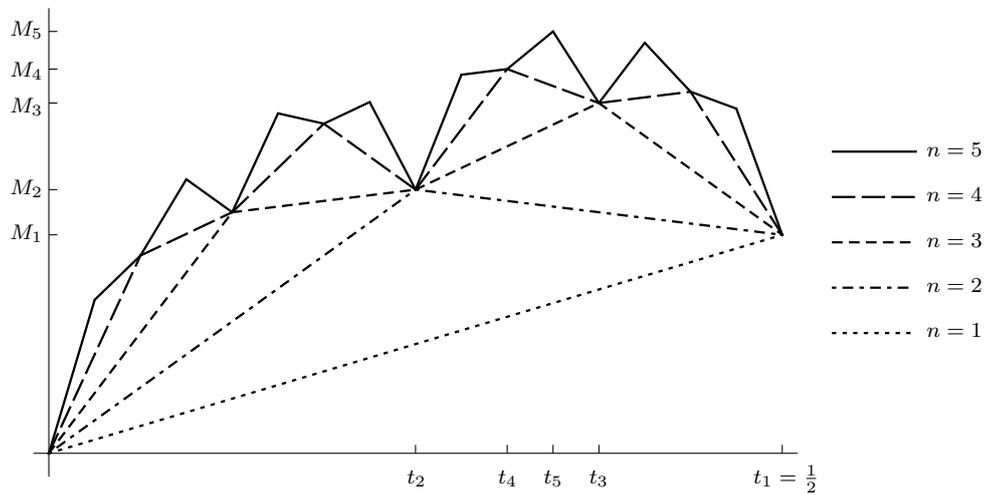
\begin{figure}[htbp]
\begin{center}
\begin{overpic}[width=12cm]{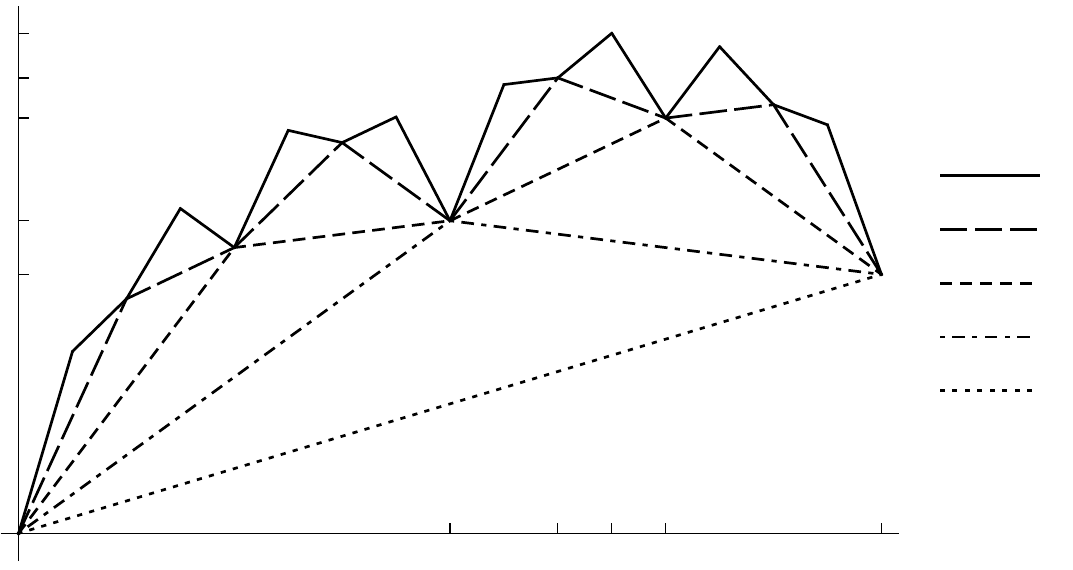}
\put(41,0){\scriptsize{$t_2$}}
\put(51,0){\scriptsize{$t_4$}}
\put(56,0){\scriptsize{$t_5$}}
\put(61,0){\scriptsize{$t_3$}}
\put(79,0){\scriptsize{$t_1=\frac12$}}
\put(-2.5,27){\scriptsize{$M_1$}}
\put(-2.5,32){\scriptsize{$M_2$}}
\put(-2.5,41){\scriptsize{$M_3$}}
\put(-2.5,45){\scriptsize{$M_4$}}
\put(-2.5,49.5){\scriptsize{$M_5$}}
\put(98,16.2){\scriptsize{$n=1$}}
\put(98,21.2){\scriptsize{$n=2$}}
\put(98,26.2){\scriptsize{$n=3$}}
\put(98,31.2){\scriptsize{$n=4$}}
\put(98,36.2){\scriptsize{$n=5$}}
\end{overpic}
\caption{Illustration of   Lemma~\ref{Jacobsthal lemma} and its proof. The functions ${\wh x}^n$ are plotted over the interval  $[0,1/2]$  for various values of $n$, together with the corresponding values for $t_n$ and $M_n$. }\label{Jacobsthal Fig}
\end{center}
\end{figure}

\begin{lemma}\label{Jacobsthal lemma}
Let 
$$J_n:=\frac13(2^n-(-1)^n)
$$
be the sequence of Jacobsthal numbers, and let 
\begin{align}\label{mN formula}
M_n:=\frac13\Big(2+\sqrt2+(-1)^{n+1}2^{-n}(\sqrt2-1)\Big)-2^{-n/2}.
\end{align}
%, and define the sequence $(a_m)$ by
%$$a_m:=\frac13(2^{m+1}+(-1)^{m})-2^{\frac{m}2-1}(1+(-1)^{m}),\qquad m=0,1,2,\dots
%$$
The function ${\wh x}^n$ has  two maximal points  given by
\begin{equation}\label{tn eq}
t_n^-:=2^{-n}J_n\in[0,1/2]\qquad\text{and}\qquad t_n^+:=1-t_n^-\in[1/2,1],
\end{equation}
these are the only maximal points, 
and the global maximum of ${\wh x}^n$ is given by
$$\max_{t\in[0,1]}{\wh x}^n(t)={\wh x}^n(t^-_n)={\wh x}^n(t^+_n)=M_n.$$
%$${\wh x}^n_{t^-_n}={\wh x}^n_{t^+_n}=2^{-n}(a_{n}+a_{n-1}\sqrt2).
%$$
\end{lemma}

\begin{proof}We first note that by the symmetry of ${\wh x}^n$ with respect to $t=1/2$ and the fact that $t_n^+=1-t_n^-$ it is sufficient to prove the result for the restriction of ${\wh x}^n$ to $[0,\frac12]$.
We will therefore simply write $t_n$ in place of $t_n^-$. 

We will prove the assertion by induction on $n$.  First, for $n=1$, we have ${\wh x}^n(t)= e_{0,0}(t)= \min\{t,1-t\}$, which is maximized at $t=1/2=2^{-1}J_1$ and has the maximum $1/2=M_1$. Given that $t_1=1/2$,  our claim that $t_n=2^{-n}J_n$ is now equivalent to the recursion
\begin{align}\label{Jacobsthal recursion}
t_{n+1}=t_n+(-1)^n2^{-(n+1)}.
\end{align}
 For $n=2$, the new maximum is taken at the peak of $e_{1,0}$, which is attained at $t=1/4=t_1-2^{-2}$. Hence  \eqref{Jacobsthal recursion} follows for $n=2$. Moreover,
 $$\max_{t\in[0,1/2]}\wh x^2(t)=\wh x^2(1/4)=\frac14+2^{-3/2}=M_2.
 $$
 Clearly, the maximizers of $\wh x^1$ and $\wh x^2$ on $[0,1/2]$ are unique. 

% For $n\ge2$, we will use induction on $n$ to prove  the recursion \eqref{Jacobsthal recursion} along with the  auxiliary claim that
%\begin{align}\label{Jabosthal slopes claim}
%|s_n^-|>|s_n^+|\text{ for   $n$ even, and }|s_n^-|<|s_n^+|\text{ for $n$ odd;}
%\end{align}
%see Figure~\ref{Jacobsthal Fig} for an illustration. 
%  For $n=2$, the new maximum is taken at the peak of $e_{1,0}$, which is attained at $t=1/4=t_1-2^{-2}$. Hence  \eqref{Jacobsthal recursion} follows for $n=2$. Since the slope  $s^-_1$ of $\wh x^1$ is equal to $1/2$ on $[0,1/2]$ and the slope of $e_{1,0}$ is $\pm\sqrt2$, we get that $s_n^-=1/2+\sqrt2$ and $s_n^+=1/2-\sqrt2$, which establishes   \eqref{Jabosthal slopes claim}
% for $n=2$.   

 Now we  assume that $n\ge2$ and that the assertion has been established for  $ n$ and $n-1$. %When defining
%$\cD_m:=\{k2^{-m}\,|\,k=0,1,\dots, 2^{m-1}\}$, we can note that 
%\begin{align}
%t_n\in\cD_n\setminus\cD_{n-1}.
%\end{align}
When letting
 \begin{align}\label{fmk eqn}
 f_{m,k}:=e_{m,k}+e_{m+1,2k}+e_{m+1,2k+1},
 \end{align}
the function  $\wh x^{n+1}$ can be obtained from
 $$\wh x^n=\wh x^{n-1}+\sum_{k=0}^{2^{n-1}-1}e_{n-1,k}
  $$ 
by replacing   all Faber--Schauder functions $e_{n-1,k}$ with the corresponding functions $f_{n-1,k}$, i.e., 
  $$\wh x^{n+1}=\wh x^{n-1}+\sum_{k=0}^{2^{n-1}-1}f_{n-1,k}.
  $$ 
  It follows from our induction hypothesis and the formula for $t_n$ that the maximum of $\wh x^n$ is attained at the peak of some function $e_{n-1,k}$. 
    Clearly, the support of $f_{n-1,k}$ coincides with the support of $e_{n-1,k}$, while the function   $\wh x^{n-1}$ is linear on that support.
  Moreover, the function $f_{n-1,k}$ has two maxima at $t_n-2^{-(n+1)} $  and $t_n+2^{-(n+1)} $, and these maxima are strictly bigger than the one of  $e_{n-1,k}$;
    see
Figure \ref{fmk fig} for an illustration. 
It therefore follows that either $\wh x^{n+1}(t_n+2^{-(n+1)})$ or $\wh x^{n+1}(t_n-2^{-(n+1)})$ must be strictly larger than $\wh x^{n}(t_n)=\max_{t}\wh x^n(t)$. Hence, the maximum of $\wh x^{n+1}$ must be attained at the peak of some Faber--Schauder function $e_{n,\ell}$ for a certain $\ell\in\{0,\dots, 2^n-1\}$. The support of  $e_{n,\ell}$  is an interval with endpoints $s$ and $s'$, and the peak of $e_{n,\ell}$ is located at $s^*=(s+s')/2$ and has height $2^{-(n+2)/2}$. The function $\wh x^{n}$, on the other hand, is linear on the support of  $e_{n,\ell}$ so that the maximum of $\wh x^{n+1}$ must satisfy
\begin{align*}
\max_{t\in[0,1/2]}\wh x^{n+1}(t)=\wh x^{n+1}(s^*)=\wh x^{n}(s^*)+e_{n,\ell}(s^*)=\frac{\wh x^n(s)+\wh x^n(s')}2+2^{-\frac{n+2}2}.
\end{align*}
At one of the endpoints of the interval enclosed by $s$ and $s'$, say at $s'$, the function $\wh x^n$ coincides with $\wh x^{n-1}$, and therefore the value $\wh x^n(s')$ can be estimated from above by $M_{n-1}$. The value $\wh x^n(s)$, on the other hand, can be estimated from above by $M_n$. We hence arrive at 
\begin{align}\label{Jacobsthal max inequality}
\wh x^{n+1}(s^*)\le\frac{M_n+M_{n-1}}2+2^{-\frac{n+2}2}=M_{n+1}.
\end{align}
But we can achieve equality in \eqref{Jacobsthal max inequality} when taking for $s^*$  the midpoint between $t_n$ and $t_{n-1}$, which is possible since, inductively by  \eqref{Jacobsthal recursion},  $t_n$ and $t_{n-1}$  enclose the interval of support of some Faber--Schauder function of generation $n$. When  making this choice, we have moreover that
$$s^*=\frac{t_n+t_{n-1}}2=\frac{t_n+t_n-(-1)^{n-1}2^{-n}}2=t_{n+1}.
$$
This proves that $\wh x^{n+1}$ has maximum $M_{n+1}$ and is maximized at $t_{n+1}$. Moreover, $t_{n+1}$ is the unique maximizer of $\wh x^{n+1}$ in $[0,1/2]$, since, by induction hypothesis, $s^*=(t_n+t_{n-1})/2$ is the only point at which equality can hold in \eqref{Jacobsthal max inequality}.
\end{proof}
 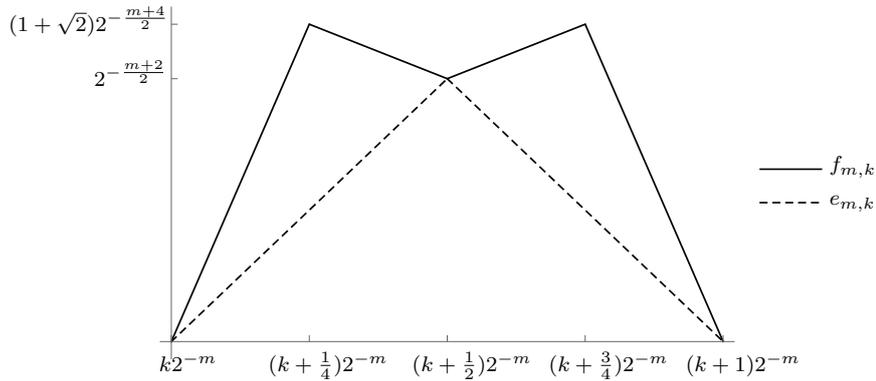
\begin{figure}[htbp]
 \begin{center}
 \begin{overpic}[width=9cm]{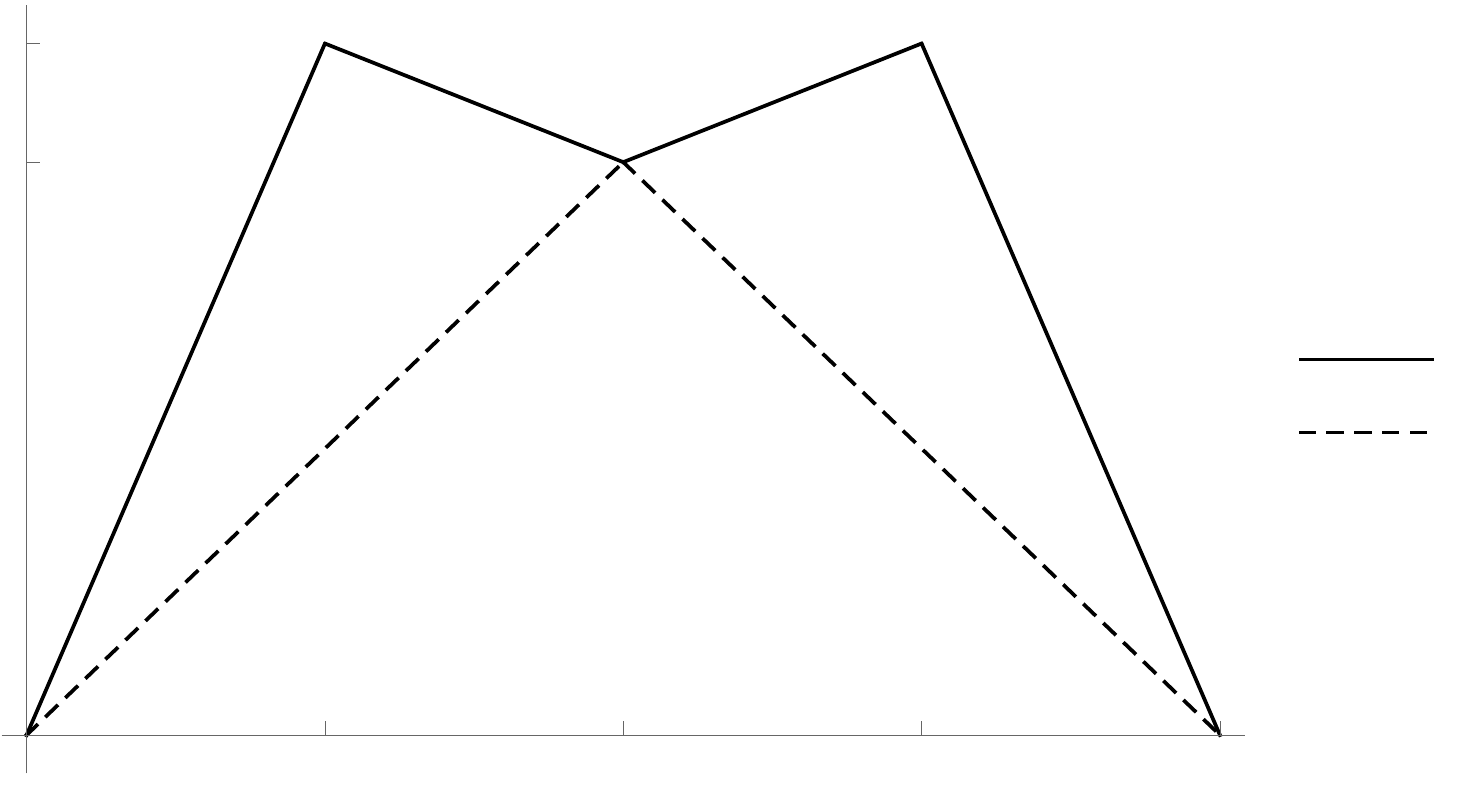}
  \put(0,0){\scriptsize{$k2^{-m} $}}
 \put(16,0){\scriptsize{$(k+\frac14)2^{-m} $}}
 \put(37,0){\scriptsize{$(k+\frac12)2^{-m} $}}
 \put(57,0){\scriptsize{$(k+\frac34)2^{-m} $}}
  \put(77,0){\scriptsize{$(k+1)2^{-m} $}}
 \put(-9.5,42){\scriptsize{$2^{-\frac{m+2}2} $}}
 \put(-22,50){\scriptsize{$(1+\sqrt2)2^{-\frac{m+4}2} $}}
 \put(98,24.5){\scriptsize{$e_{m,k} $}}
 \put(98,29.5){\scriptsize{$f_{m,k} $}}
 \end{overpic}
  \end{center}
   \caption{The function $f_{m,k}$ defined in \eqref{fmk eqn}, plotted alongside $e_{m,k}$.}
 \label{fmk fig}
 \end{figure}

\noindent{\it Proof of Theorem~\ref{max osc thm}.} Part (a): For~$t$ fixed, $x(t)=\sum_{m=0}^{\infty}\sum_{k=0}^{2^m-1}\theta_{m,k}e_{m,k}(t)$ is maximized by taking all coefficients $\theta_{m,k}$ equal to~$+1$, i.e., $|x(t)|\le|\wh x(t)|$ for all $x\in\cX$ and $t\in[0,1]$. The statement on maximum and maximizers of $\wh x$ 
 follows immediately from Lemma~\ref{Jacobsthal lemma} by  noting that $M_n\to\frac13(2+\sqrt2)$ and $t_n^-\to\frac13$ as $n\ua\infty$. 
 
 Part (b): We first show that the maximal oscillation of an arbitrary $x\in\cX$ is dominated by the maximal oscillation of $x^*$. See Figure \ref{xstar fig} for a plot of $x^*$. To this end, we may assume without loss of generality that the coefficient $\theta_{0,0}$ in the Faber--Schauder development of $x$ is equal to $+1$. Next, we note  that for $0\le s\le1/2$, \begin{align}
 x(s)\le\wh x(s)=x^*(s) 
  \end{align}
  and
  \begin{align} x(s)\ge e_{0,0}(s)-\big(\wh x(s)-e_{0,0}(s)\big)=2e_{0,0}(1-s)-\wh x(1-s)=x^*(1-s).
 \end{align}
 For $1/2\le t\le1$, we get in the same manner that 
 \begin{align}
 x^*(t)\le x(t)\le x^*(1-t).
 \end{align}
 It follows that
 \begin{align}\label{separation identity}
\max_{s,t\in[0,1]}|x(t)-x(s)|\le\max_{s\in[0,1/2]}\max_{ t\in[1/2,1]}|x^*(t)-x^*(s)|,
\end{align}
and by taking $x=x^*$ we see that the right-hand side is in fact  equal to $\max_{s,t\in[0,1]}|x^*(t)-x^*(s)|$.
 
 To compute the right-hand side of \eqref{separation identity}, note that we have $x^*=\wh x$ on $[0,1/2]$, and so the maximum of $x^*$ is attained in $s=1/3$ with value $\frac13(2+\sqrt2)$, due to the first part of this theorem. Moreover, on $[1/2,1]$ we have that $-x^*(t)=\wh x({t-1/2})-1/2$. It therefore follows from part (a) of the theorem that the minimum of $x^*$ is attained at $t=5/6$ with minimal value $\frac12-\frac13(2+\sqrt2)$; see Figure \ref{xstar fig}. Therefore,
$$\max_{s\in[0,1/2]}\max_{ t\in[1/2,1]}|x^*(t)-x^*(s)|=\frac23(2+\sqrt2)-\frac12=\frac56+\frac{\sqrt8}3.\eqno{\qedsymbol}
$$

\begin{figure}[htbp]
 \begin{center}
\begin{overpic}[width=9cm]{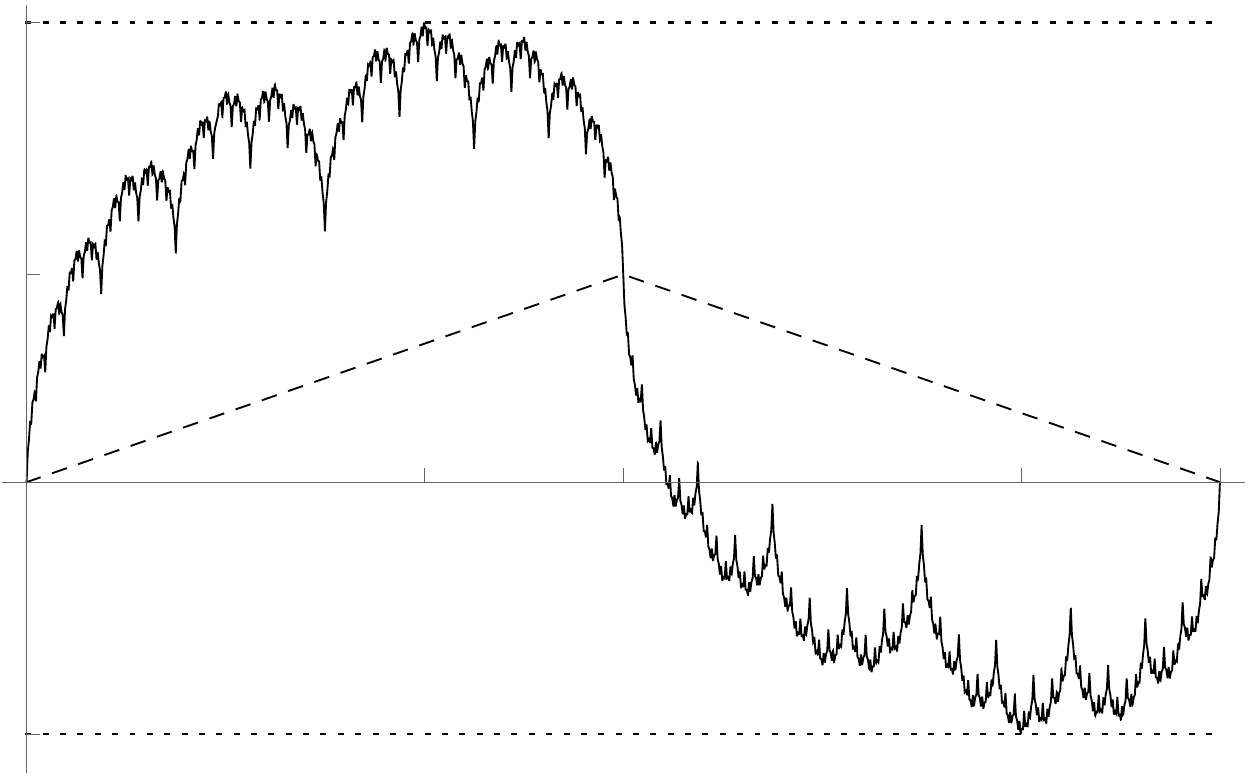}
  \put(0,21){\scriptsize{$0 $}}
  \put(31.6,21){\scriptsize{$1/3 $}}
  \put(47,21){\scriptsize{$1/2 $}}
   \put(79,21){\scriptsize{$5/6 $}}
   \put(-20,2.5){\scriptsize{$\frac12-\frac13(2+\sqrt2) $}}
  \put(-0.5,39.5){\scriptsize{$\frac12$}}
  \put(-13.5,59.5){\scriptsize{$\frac13(2+\sqrt2) $}}
\end{overpic}
  \end{center}
   \caption{The function $x^*$ with its maximum and minimum, and the function $e_{0,0}$ (dashed).}
 \label{xstar fig}
 \end{figure}
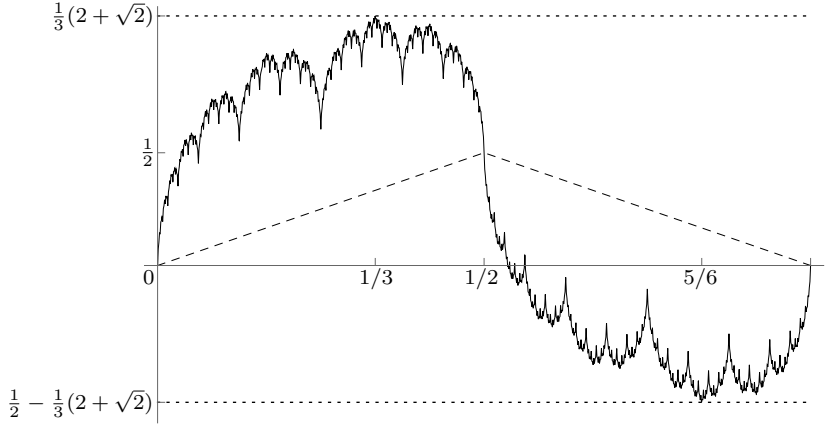

\subsection{Proofs of Theorem~\ref{modulus continuity thm} and of Corollary~\ref{compact cor}}

Let $h>0$ be given, and define $n:=\nu(h)$ so that $2^{-(n+1)}<h\le 2^{-n}$.
   Faber--Schauder functions  $e_{m,k}$ can be linear with slope $\pm2^{m/2}$ on intervals of length  $2^{-(m+1)}$. Using this fact for $m\le n-2$ and for the interval $[t,t+2^{-n}]$, we get
  \begin{align}\label{mod cont initial estimate}
|x(t+h)-x(t)|&\le \sum_{m=0}^{n-2}h2^{m/2}+\bigg|\sum_{m=n-1}^\infty\sum_{k=0}^{2^m-1}\theta_{m,k}\Big(e_{m,k}(t+h)-e_{m,k}(t)\Big)\bigg|,
\end{align}
where $\theta_{m,k}$ are the coefficients in the Faber--Schauder development of $x$.
The left-hand sum is equal to $h(1+\sqrt2)(2^{(n-1)/2}-1)$. 
This observation suffices in the situation of~\cite{HataYamaguti}  and~\cite{Allaartflexible} to determine the corresponding moduli of continuity, because in these papers the Faber--Schauder coefficients are such that the sum on the right-hand side can be asymptotically neglected. 

In our case, the sum on the right-hand side of \eqref{mod cont initial estimate} will turn out to be of the same order as $h2^{\nu(h)}$ and therefore cannot be neglected. To deal with it, we note that the Faber--Schauder functions have the following scaling properties:
\begin{align}\label{Schauder self-similarity}
\sqrt 2 e_{m,k}(t)=e_{{m-1},k}(2t)\qquad \text{and}\qquad e_{m,k}(t-\ell2^{-m})=e_{m,k+\ell}(t)
\end{align}
for $t\in\bR$, $m\ge0$, and $k,\ell\in\bZ$. 
For $t\in[0,1-2^{-n})$ given, let $\ell\in\bN_0$ be such that $\ell 2^{-n}\le t<(\ell+1)2^{-n}$ and define $s:=t-\ell 2^{-n}$. Then  $ 0\le  2^{n-1}s<1/2$ and $ 1/4<2^{n-1}(s+h)<1$. The scaling properties \eqref{Schauder self-similarity}
 imply that for $m\ge n$,
 \begin{align}\label{Faber-Schauder rescaling}
 e_{m,k}(t+h)-e_{m,k}(t)=2^{\frac{1-n}2}\big(e_{m-(n-1), k-\ell2^{m-n}}( 2^{n-1}(s+h))-e_{m-(n-1), k-\ell2^{m-n}}( 2^{n-1}s)\big).
 \end{align}
The case $m=n-1$ needs additional care, depending on whether $\ell$ is  even or odd. 

{\it Case 1: $\ell$ is even.} In this case, $[\ell2^{-n},(\ell+2)2^{-n}]$ contains both $t$ and $t+h$ and is equal to the support of $e_{n-1,\ell/2}$. Therefore, the  identity \eqref{Faber-Schauder rescaling} extends to $m=n-1$  and  we arrive at 
\begin{align}\label{Schauder sum rescaling eq}
\sum_{m=n-1}^\infty\sum_{k=0}^{2^m-1}\theta_{m,k}\Big(e_{m,k}(t+h)-e_{m,k}(t)\Big)=2^{\frac{1-n}2}\big(y( 2^{n-1}(s+h)))-y(2^{n-1}s)\big)
\end{align}
for
\begin{align}\label{y case 1}
y:=
\sum_{m=0}^\infty\sum_{k=0}^{2^m-1}\theta_{m+n-1,k+\ell 2^{m-1}}e_{m,k}\in\cX.
\end{align}

{\it Case 2: $\ell$ is odd.} In this case, the interval $[\ell2^{-n},(\ell+1)2^{-n}]$ contains $t$ and belongs to the support of $e_{n-1,(\ell-1)/2}$, while the interval $[(\ell+1)2^{-n},(\ell+2)2^{-n}]$ belongs to the support of $e_{n-1,(\ell+1)/2}$. The point $t+h$ may belong to either of the two intervals. We need to distinguish two further cases, depending on whether the coefficients $\theta_{n-1,(\ell-1)/2}$ and $\theta_{n-1,(\ell+1)/2}$ are identical or different.

{\it Case 2a: $\theta_{n-1,(\ell-1)/2}=\theta_{n-1,(\ell+1)/2}$.}  We have 
$$e_{n-1,(\ell-1)/2}(r)+e_{n-1,(\ell+1)/2}(r)=2^{\frac{1-n}2}\Big(\frac12-e_{0,0}(2^{{n-1}}r-\ell/2)\Big)\qquad\text{for $r\in[\ell2^{-n},(\ell+2)2^{-n}]$.}
$$
We thus arrive  again at \eqref{Schauder sum rescaling eq}, but this time for 
$$
y:=-\theta_{n-1,(\ell-1)/2} e_{0,0}+
\sum_{m=1}^\infty\sum_{k=0}^{2^m-1}\theta_{m+n-1,k+\ell 2^{m-1}}e_{m,k}\in\cX.$$

{\it Case 2b: $\theta_{n-1,(\ell-1)/2}=-\theta_{n-1,(\ell+1)/2}$.} Since $e_{n-1,(\ell-1)/2}(r)-e_{n-1,(\ell+1)/2}(r)$
 is linear on $[\ell2^{-n},(\ell+2)2^{-n}]$ with slope $- 2^{\frac{n-1}2}$,  we get
 \begin{equation}
\label{Case 2b eq} 
\begin{split}
 |x(t+h)-x(t)|&\le \sum_{m=0}^{n-1}h2^{m/2}+\bigg|\sum_{m=n}^\infty\sum_{k=0}^{2^m-1}\theta_{m,k}\Big(e_{m,k}(t+h)-e_{m,k}(t)\Big)\bigg|\\
 &=h(1+\sqrt2)(2^{n/2}-1)+2^{-n/2}|y( 2^{n}(s+h)))-y(2^{n}s)|,
 \end{split}
 \end{equation}
where
\begin{align}\label{Case 2b y}
y(r)=\begin{cases}
y_0(r)&\text{if $0\le r\le 1$,}\\
y_1(r-1)&\text{if $1\le r\le2$,}
\end{cases}
\end{align}
and $y_0,y_1\in\cX$ are given by
\begin{align}
y_0:=\sum_{m=0}^\infty\sum_{k=0}^{2^{m}-1}\theta_{m+n,k+\ell 2^{m}}e_{m,k},\qquad y_1:=\sum_{m=0}^\infty\sum_{k=0}^{2^{m}-1}\theta_{m+n,k+(\ell+1) 2^{m}}e_{m,k}
\end{align}
See Figure \ref{xlowerstar fig} for an illustration in which $y_0=\wh x=-y_1$, as it will be needed in the proof of Theorem~\ref{modulus continuity thm} (b).

\medskip

\begin{proof}[Proof of Theorem~\ref{modulus continuity thm} (a)] Let again $h>0$ be given with $n=\nu(h)$. Since $\theta_{m,k}=1$ for all $m,k$, \eqref{mod cont initial estimate}, \eqref{Schauder sum rescaling eq}, \eqref{y case 1}, and Theorem~\ref{max osc thm} (a) 
 imply that in Case 1
 \begin{align}|\wh x(t+h)-\wh x(t)|&\le h(1+\sqrt2)(2^{(n-1)/2}-1)+2^{\frac{1-n}2}\big|\wh x( 2^{n-1}(s+h)))-\wh x(2^{n-1}s)\big|\label{wh x Case 1 estimate}\\
 &\le h(1+\sqrt2)2^{(n-1)/2}+\frac{2^{\frac{1-n}2}}3(2+\sqrt2)=\om (h).\nonumber
 \end{align}
In Case 2a we get a similar estimate, but on the right-hand side of \eqref{wh x Case 1 estimate}, $\wh x$ needs to be replaced by 
$$\wh y:=- e_{0,0}+
\sum_{m=1}^\infty\sum_{k=0}^{2^m-1}e_{m,k}.
$$
But note that $\wh y(t)=\wh x(\varphi(t))-1/2$, where $\varphi(t)=1/2-t$ for $t\le1/2$ and $\varphi(t)=t-1/2$ for $t\ge1/2$. Therefore, $|\wh x(t+h)-\wh x(t)|\le\om (h)$ also holds in Case 2a. Finally, Case 2b cannot occur. This concludes the proof of \lq\lq$\le$\rq\rq.

To prove  \lq\lq$\ge$\rq\rq, we let $t:=0$ and $h_n:=\frac23 2^{-n}$. Then we are in Case 1, and 
$$|\wh x(t+h_n)-\wh x(t)|=(1+\sqrt2)(2^{(n-1)/2}-1)h_n+2^{(1-n)/2}\wh x(2^{n-1}h_n)=\om (h_n)-(1+\sqrt2)h_n,
$$
by Theorem~\ref{max osc thm} (a). Since $\om (h_n)=O(\sqrt{h_n})$, we get \lq\lq$\ge$\rq\rq.\end{proof}

\begin{proof}[Proof of Theorem~\ref{modulus continuity thm} (b)] Let again $h>0$ be given with $n=\nu(h)$. In Case 1, \eqref{mod cont initial estimate}, \eqref{Schauder sum rescaling eq}, \eqref{y case 1}, and Theorem~\ref{max osc thm} (b) 
 imply that for each $x\in\cX$
  \begin{align}| x(t+h)- x(t)|&\le h(1+\sqrt2)(2^{(n-1)/2}-1)+2^{\frac{1-n}2}\sup_{y\in\cX}\big|y( 2^{n-1}(s+h)))-y(2^{n-1}s)\big|\nonumber\\
 &\le h(1+\sqrt2)2^{(n-1)/2}+\frac{2^{\frac{1-n}2}}6(5+4{\sqrt2})\nonumber\\ 
 &=(1+\sqrt2)h2^{n/2}-\frac{2^{\log_2h+n/2}}{\sqrt2} +\frac13\Big(\frac5{\sqrt2}+4\Big)2^{-n/2}. \label{sqrt2om domination}\end{align}
 Since $\log_2h+n/2\ge-n/2-1$, the middle term of the preceding sum is bounded from above by $-2^{-n/2}/\sqrt8$, and so the entire sum in \eqref{sqrt2om domination} is dominated by
 $$(1+\sqrt2)h2^{n/2}+\frac13\Big(\frac5{\sqrt2}+4-\frac3{\sqrt8}\Big)2^{-n/2},
 $$
 which is in turn   dominated by $\sqrt2\om(h)$.

  The same inequality as in Case 1 holds in Case 2a. 
  In Case 2b, we obtain 
    \begin{align*}| x(t+h)- x(t)|&\le    h(1+\sqrt2)(2^{n/2}-1)+2^{-n/2}\sup_{0\le r\le 2}\sup_{0\le s\le 2}|y( r)-y(s)|    ,
     \end{align*}
     where $y$ is as in \eqref{Case 2b y} for $y_0,y_1\in\cX$. Clearly, the supremum on the right-hand side is maximized when  $y_0=-y_1=\wh x$ and then equal to $\frac23(2+\sqrt2)$ according to Theorem~\ref{max osc thm} (a). Therefore,
     $$| x(t+h)- x(t)|\le  h(1+\sqrt2)2^{n/2}+2^{-n/2}\frac23(2+\sqrt2)=\sqrt2\om(h). 
     $$
      This concludes the proof of \lq\lq$\le$\rq\rq.

To prove  \lq\lq$\ge$\rq\rq, we let
$$t_n:=\frac12 -\frac13 2^{-n} \qquad \text{and}\qquad h_n:=\frac232^{-n},
$$
so that $\nu(h_n)=n$. When considering the function $x_*:=-x^*\in\cX$, we are  in Case 2b. 
First, we have $e_{0,0}(t_n+h_n)-e_{0,0}(t_n)=0$. Second,  for $m=1,\dots, n$, the function
$$g_m:=-e_{m,2^{m-1}-1}+e_{m,2^{m-1}}
$$
is linear on $[\frac12-2^{-m-1},\frac12+2^{-m-1}]\supset[t_n,t_n+h_n]$ with slope $2^{m/2}$. Using this fact for $m\le n-1$ gives
\begin{align*}
x_*(t_n+h_n)-x_*(t_n)&=\sum_{m=1}^{n-1}2^{m/2}h_n+\sum_{m=n}^\infty\bigg(\sum_{\ell=2^{m-1}}^{2^m-1}e_{m,\ell}(t_n+h_n)+\sum_{k=0}^{2^{m-1}-1}e_{m,k}(t_n)\bigg)\\
&=(1+\sqrt2)(2^{n/2}-\sqrt2)h_n+2^{-n/2}\big(\wh x(1/3)+\wh x(2/3)\big)\\
& =(1+\sqrt2)2^{n/2}h_n-(\sqrt 2+2)h_n+2^{-n/2}\frac23(2+\sqrt2)\\
&=\sqrt2\om(h_n)-(\sqrt 2+2)h_n,
\end{align*}
where, in the second step, we have argued as in \eqref{Case 2b eq}, and, in the third step, we have used Theorem~\ref{max osc thm} (a); see Figure \ref{xlowerstar fig} for an illustration.  Since $\om (h_n)=O(\sqrt{h_n})$, we get \lq\lq$\ge$\rq\rq.
\end{proof}
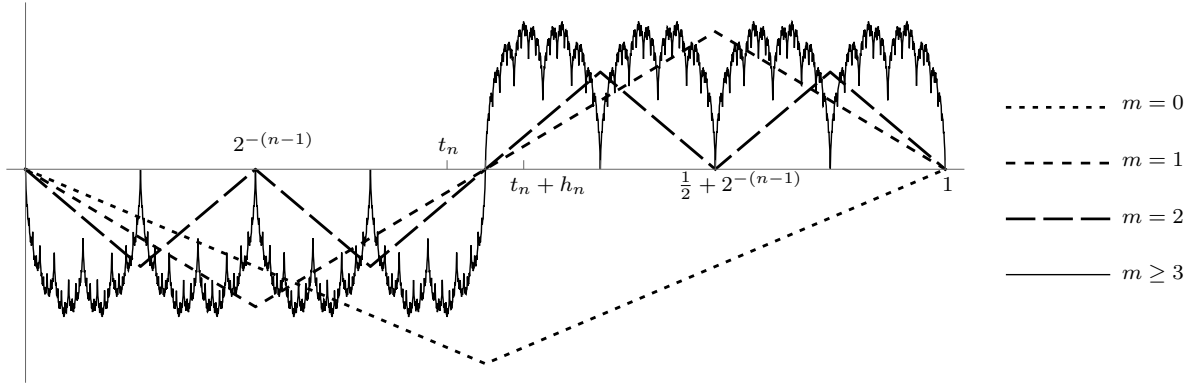
\begin{figure}
\begin{center}
\begin{overpic}[width=15cm]{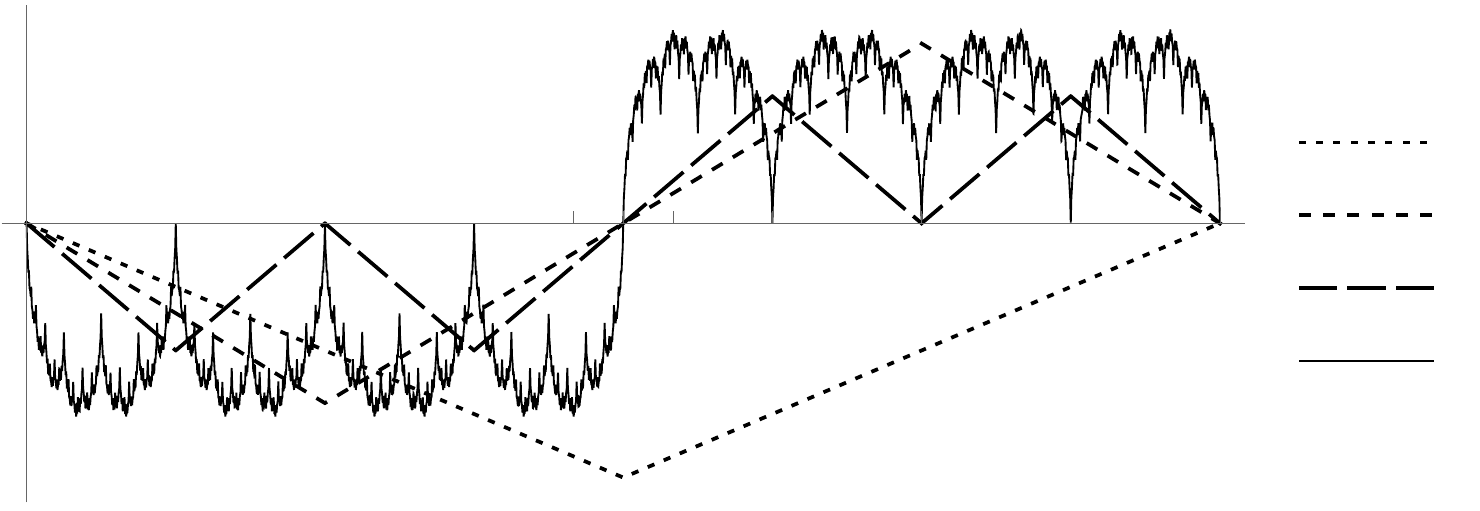}
%\put(42,17.5){\scriptsize{$\frac12$}}
\put(20,21){\scriptsize{$2^{-(n-1)}$}}
\put(38,21){\scriptsize{$t_n$}}
\put(44.3,17.5){\scriptsize{$t_n+h_n$}}
%\put(51.5,17.5){\scriptsize{$\frac12+2^{-n}$}}
\put(59,17.5){\scriptsize{$\frac12+2^{-(n-1)}$}}
\put(82.2,17.5){\scriptsize{$1$}}
\put(98,9.7){\scriptsize{$m\ge3$}}
\put(98,14.7){\scriptsize{$m=2$}}
\put(98,19.7){\scriptsize{$m=1$}}
\put(98,24.7){\scriptsize{$m=0$}}
\end{overpic}
\caption{Illustration of Case 2b and, in particular, of the proof of \lq$\ge$\rq\ in Theorem~\ref{modulus continuity thm} (b) for $n=3$. The functions $e_{m,k}$ in the Faber--Schauder development of $x_*=-x^*$ are plotted individually for  generations $m\le n-1$. Their sum   over all generations $m\ge n$ corresponds to a sequence of rescaled functions $\wh x$.}
\label{xlowerstar fig}
\end{center}
\end{figure}

\medskip

\begin{proof}[Proof of Corollary~\ref{compact cor}]   Theorem~\ref{max osc thm}  implies that the family $\cX$ is uniformly bounded and  Theorem~\ref{modulus continuity thm} (b) yields that $\cX$ is equicontinuous. Therefore it only remains to show that $\cX$ is closed in $C[0,1]$. Following Hata and Yamaguti~\cite{HataYamaguti}, let $(x_n)_{n\in\bN}$ be a sequence in $\cX$ that converges uniformly to some $x\in C[0,1]$. We clearly have $x(0)=x(1)=0$. It is well known that any such function can be uniquely represented as a uniformly convergent  series of Faber--Schauder functions,
\begin{align}\label{Faber-Schauder representation of cont funct}
x=\sum_{m=0}^\infty\sum_{k=0}^{2^m-1}a_{m,k}(x)e_{m,k},
\end{align}
where the coefficients $a_{m,k}$ are given as 
$$a_{m,k}(x)=2^{m/2}\bigg(2x\Big(\frac{2k+1}{2^{m+1}}\Big)-x\Big(\frac k{2^{m}}\Big)-x\Big(\frac{k+1}{2^{m}}\Big)\bigg).
$$
Clearly, we have $a_{m,k}(x_n)\to a_{m,k}(x)$ for all pairs $m,k$ as $n\ua\infty$. But since each $x_n\in\cX$ and the representation \eqref{Faber-Schauder representation of cont funct} is unique, we must have $a_{m,k}(x_m)\in\{-1,+1\}$, which implies that also $a_{m,k}(x)\in\{-1,+1\}$ and in turn that $x\in\cX$.
\end{proof}

\subsection{Proofs of Propositions~\ref{Takagi quad var prop} and~\ref{MountTakagi_var_prop}}

\begin{remark}\label{BVQVRemark}
As already observed in \cite[Remark 8]{SchiedCPPI}, the following facts  follow  easily from Propositions 2.2.2, 2.2.9, and 2.3.2 in \cite{Sondermann}.  Suppose that $x\in C[0,1]$  admits the continuous quadratic variation $\<x\>$ along $(\bT_n)$ and $f$ is continuous and of finite variation. Then both $\<f\>$ and $\<x+f\>$ exist  along $(\bT_n)$  and are given by $\<f\>=0$ and $\<x+f\>=\<x\>$. By means of the polarization identity \eqref{polarization in approx} we get moreover that $\<x,f\>=0$.
\end{remark}

\begin{proof}[Proof of Proposition~\ref{Takagi quad var prop}]  Lemma 1.1 (ii) in \cite{Gantert} states  that a function $x\in C[0,1]$ with Faber--Schauder development $ \sum_{m=0}^{\infty}\sum_{k=0}^{2^m-1}\theta_{m,k}e_{m,k}$ satisfies
$$\<x\>_1^n=\frac1{2^n}\sum_{m=0}^{n-1}\sum_{k=0}^{2^m-1}\theta_{m,k}^2.
$$ 
This immediately yields $\<x\>_1=1$ for all $x\in\cX$. 

The first scaling property in \eqref{Schauder self-similarity} implies that for any $x\in\cX$ there exists $y\in\cX$ and a linear function $f$ such that $x(t)=f(2t)+2^{-1/2}y(2t)$ for $0\le t\le 1/2$. It hence follows from Remark \ref{BVQVRemark} that $\<x\>_{\frac12}=\<f+2^{-1/2}y\>_1=\<2^{-1/2}y\>_1=1/2$. Iteratively, we obtain $\<x\>_{2^{-\ell}}=2^{-\ell}$ for all $\ell\in\bN$.  Using also the second scaling property in \eqref{Schauder self-similarity} gives in a similar manner that $\<x\>_{(k+1)2^{-\ell}}-\<x\>_{k2^{-\ell}}=2^{-\ell}$ for  $k,\ell\in\bN$ with $(k+1)2^{-\ell}\le1$. We therefore arrive at $\<x\>_t=t$ for all dyadic rationals $t\in[0,1]$. A sandwich argument extends this fact to all $t\in[0,1]$.\end{proof}

 \medskip
 
 For $x\in\cX$ with Faber--Schauder expansion
$x= \sum_{m=0}^{\infty}\sum_{k=0}^{2^m-1}\theta_{m,k}e_{m,k}
$
and $n\in\bN$, we define
$x^n$ by \eqref{Mount Takagi approx xN}. 

\medskip

\begin{proof}[Proof of Proposition 
\ref{MountTakagi_var_prop}] We first show \eqref{MountTakagi_var_prop eq2} for $t=1$.  To this end, we  note that $x(s)=x^{n+k}(s)$ for $s\in\bT_n$ and $k\ge0$. Hence,
\begin{align}\label{MountTakagi_var_prop covar approx eq}
\<\wh x,y\>^n_1:=\sum_{s \in\bT_n }(\wh x(s')-\wh x(s))(y(s')-y(s))=\sum_{s \in\bT_n }(\wh x^n(s')-\wh x^n(s))(y^n(s')-y^n(s)).
\end{align}
Moreover,  
$$ \wh x^{n+1}\Big(\frac{s+s'}2\Big)=\wh x^n\Big(\frac{s+s'}2\Big)+\Delta_{n+1}=\frac12(\wh x^n(s)+\wh x^n(s'))+ \Delta_{n+1},
$$
where  $\Delta_{n+1}=2^{-(n+2)/2}$ is the maximal amplitude of a Faber--Schauder function $e_{n,\ell}$. Similarly, for $y$ we have
$$ y^{n+1}\Big(\frac{s+s'}2\Big)=y^n\Big(\frac{s+s'}2\Big)+(-1)^{n+1}\Delta_{n+1}=\frac12(y^n(s)+y^n(s'))+ (-1)^{n+1}\Delta_{n+1}.
$$
Therefore, when passing from $n$ to $n+1$ in \eqref{MountTakagi_var_prop covar approx eq}, each term $(\wh x^n(s')-\wh x^n(s))(y^n(s')-y^n(s))$  will be replaced by 
\begin{align*}\Big(\wh x^{n+1}(s')-&\wh x^{n+1}\Big(\frac{s+s'}2\Big)\Big)\Big(y^{n+1}(s')-y^{n+1}\Big(\frac{s+s'}2\Big)\Big)\\
+\Big(\wh x^{n+1}\Big(&\frac{s+s'}2\Big)-\wh x^{n+1}(s)\Big)\Big(y^{n+1}\Big(\frac{s+s'}2\Big)-y^{n+1}(s)\Big)\\
&=\Big(\frac12\big(\wh x^n(s')-\wh x^n(s)\big)-\Delta_{n+1}\Big)\Big(\frac12\big(y^n(s')-y^n(s)\big)-(-1)^{n+1}\Delta_{n+1}\Big)\\
&\qquad+\Big(\frac12\big(\wh x^n(s')-\wh x^n(s)\big)+\Delta_{n+1}\Big)\Big(\frac12\big(y^n(s')-y^n(s)\big)+(-1)^{n+1}\Delta_{n+1}\Big)\\
&=\frac12(\wh x^n(s')-\wh x^n(s))(y^n(s')-y^n(s))+2(-1)^{n+1}\Delta_{n+1}^2\\
&=\frac12(\wh x^n(s')-\wh x^n(s))(y^n(s')-y^n(s))+(-1)^{n+1}2^{-n-1}
\end{align*}
So
$$\<\wh x,y\>^{n+1}_1=\frac12 \<\wh x,y\>^n_1+ (-1)^{n+1}\sum_{s\in\bT_n}2^{-n-1}=\frac12 \Big(\<\wh x,y\>^n_1+(-1)^{n+1}\Big) 
$$
and in turn
$$\<\wh x,y\>^{2n+1}_1=\frac14\<\wh x,y\>^{2n-1}_1+\frac14\qquad\text{and}\qquad \<\wh x,y\>^{2n+2}_1=\frac14\<\wh x,y\>^{2n}_1-\frac14.
$$
This recursion easily implies  \eqref{MountTakagi_var_prop eq2} for $t=1$. 

Next, still for $t=1$, the identities \eqref{MountTakagi_var_prop eq1} follow immediately from \eqref{MountTakagi_var_prop eq2} and the polarization identity \eqref{polarization in approx}.   To prove \eqref{MountTakagi_var_prop eq1} for all $t\in[0,1]$, we note that the Faber--Schauder expansion \eqref{Mount_Takagi_moQV eq} of $\wh y:=\wh x+y$ and the first scaling property in \eqref{Schauder self-similarity} imply the following self-similarity relation: $ \wh y(t)=f(4t)+\frac12\wh y(4t)$, where $0\le t\le 1/4$ and $f$ is a piecewise linear function and hence of bounded variation. It therefore follows as in the proof of Proposition~\ref{Takagi quad var prop} that, if the quadratic variation $\<\wh y\>_1$ exists along some subsequence of $(\bT_n)$, then  $\<\wh y\>_{\frac14}$ also exists along that subsequence and equals $\frac14\<\wh y\>_1$. The two identities \eqref{MountTakagi_var_prop eq1} now follow  as in the proof of Proposition~\ref{Takagi quad var prop} and by further exploiting the self-similarity of $\wh y$. By using once again polarization \eqref{polarization in approx}, we finally arrive at \eqref{MountTakagi_var_prop eq2} for all $t\in[0,1]$.
\end{proof}

\medskip

\noindent {\bf Acknowledgement.} The author expresses his gratitude toward two anonymous referees for valuable comments and suggestions that  helped to improve the paper substantially. 

\bibliography{CTBook}{}

\begin{thebibliography}{10}

\bibitem{Allaartflexible}
P.~C. Allaart.
\newblock On a flexible class of continuous functions with uniform local
  structure.
\newblock {\em Journal of the Mathematical Society of Japan}, 61(1):237--262,
  2009.

\bibitem{AllaartKawamura}
P.~C. Allaart and K.~Kawamura.
\newblock The {T}akagi function: a survey.
\newblock {\em Real Analysis Exchange}, 37(1):1--54, 2011.

\bibitem{Benderetal1}
C.~Bender, T.~Sottinen, and E.~Valkeila.
\newblock Pricing by hedging and no-arbitrage beyond semimartingales.
\newblock {\em Finance Stoch.}, 12(4):441--468, 2008.

\bibitem{BickWillinger}
A.~Bick and W.~Willinger.
\newblock Dynamic spanning without probabilities.
\newblock {\em Stochastic Process. Appl.}, 50(2):349--374, 1994.

\bibitem{Billingsley1982}
P.~{Billingsley}.
\newblock {Van der Waerden's continuous nowhere differentiable function.}
\newblock {\em {Am. Math. Mon.}}, 89:691, 1982.

\bibitem{ContFournieJFA}
R.~Cont and D.-A. Fourni{\'e}.
\newblock Change of variable formulas for non-anticipative functionals on path
  space.
\newblock {\em J. Funct. Anal.}, 259(4):1043--1072, 2010.

\bibitem{ContFournieAOP}
R.~Cont and D.-A. Fourni{\'e}.
\newblock Functional {I}t\^o calculus and stochastic integral representation of
  martingales.
\newblock {\em Ann. Probab.}, 41(1):109--133, 2013.

\bibitem{DavisRavalObloij}
M.~Davis, J.~Ob{\l}{\'o}j, and V.~Raval.
\newblock Arbitrage bounds for prices of weighted variance swaps.
\newblock {\em {\rm To appear in} Mathematical Finance}, 2014.

\bibitem{deRham}
G.~de~Rham.
\newblock Sur un exemple de fonction continue sans d{\'e}riv{\'e}e.
\newblock {\em Enseign. Math}, 3:71--72, 1957.

\bibitem{Dupire}
B.~Dupire.
\newblock Functional {I}t{\^o} calculus.
\newblock {\em Bloomberg Portfolio Research Paper No. 2009-04-FRONTIERS}, 2009.
\newblock Available at {\tt http://dx.doi.org/10.2139/ssrn.1435551}.

\bibitem{EkrenKellerTouzi}
I.~Ekren, C.~Keller, N.~Touzi, and J.~Zhang.
\newblock On viscosity solutions of path dependent {PDE}s.
\newblock {\em Ann. Probab.}, 42(1):204--236, 2014.

\bibitem{Faber}
G.~Faber.
\newblock {\"U}ber die {O}rthogonalfunktionen des {H}errn {H}aar.
\newblock {\em Jahresbericht der deutschen Mathematiker-Vereinigung},
  19:104--112, 1910.

\bibitem{FoellmerIto}
H.~F{\"o}llmer.
\newblock Calcul d'{I}t\^o sans probabilit\'es.
\newblock In {\em Seminar on {P}robability, {XV} ({U}niv. {S}trasbourg,
  {S}trasbourg, 1979/1980)}, volume 850 of {\em Lecture Notes in Math.}, pages
  143--150. Springer, Berlin, 1981.

\bibitem{FoellmerECM}
H.~F{\"o}llmer.
\newblock Probabilistic aspects of financial risk.
\newblock In {\em European {C}ongress of {M}athematics, {V}ol. {I}
  ({B}arcelona, 2000)}, volume 201 of {\em Progr. Math.}, pages 21--36.
  Birkh\"auser, Basel, 2001.

\bibitem{FoellmerSchiedBernoulli}
H.~F{\"o}llmer and A.~Schied.
\newblock Probabilistic aspects of finance.
\newblock {\em Bernoulli}, 19(4):1306--1326, 2013.

\bibitem{Freedman}
D.~Freedman.
\newblock {\em Brownian motion and diffusion}.
\newblock Springer-Verlag, New York, second edition, 1983.

\bibitem{FrizHairer}
P.~K. Friz and M.~Hairer.
\newblock {\em A course on rough paths}.
\newblock Springer-Verlag, Heidelberg, 2014.

\bibitem{GantertDiss}
N.~Gantert.
\newblock {\em Einige grosse {A}bweichungen der {B}rownschen {B}ewegung}.
\newblock Dissertation, Rheinische Friedrich-Wilhelms-Universit{\"a}t Bonn;
  Bonner Mathematische Schriften, 224. 1991.

\bibitem{Gantert}
N.~Gantert.
\newblock Self-similarity of {B}rownian motion and a large deviation principle
  for random fields on a binary tree.
\newblock {\em Probab. Theory Related Fields}, 98(1):7--20, 1994.

\bibitem{HataYamaguti}
M.~Hata and M.~Yamaguti.
\newblock The {T}akagi function and its generalization.
\newblock {\em Japan J. Appl. Math.}, 1(1):183--199, 1984.

\bibitem{Kahane}
J.-P. Kahane.
\newblock Sur l'exemple, donn{\'e} par {M}. de {R}ham, d'une fonction continue
  sans d{\'e}riv{\'e}e.
\newblock {\em Enseignement Math}, 5:53--57, 1959.

\bibitem{KaratzasShreve}
I.~Karatzas and S.~E. Shreve.
\newblock {\em Brownian motion and stochastic calculus}, volume 113 of {\em
  Graduate Texts in Mathematics}.
\newblock Springer-Verlag, New York, second edition, 1991.

\bibitem{Kono}
N.~K{\^o}no.
\newblock On generalized {T}akagi functions.
\newblock {\em Acta Math. Hungar.}, 49(3-4):315--324, 1987.

\bibitem{Lagarias}
J.~C. Lagarias.
\newblock The {T}akagi function and its properties.
\newblock In {\em Functions in number theory and their probabilistic aspects},
  RIMS K\^oky\^uroku Bessatsu, B34, pages 153--189. Res. Inst. Math. Sci.
  (RIMS), Kyoto, 2012.

\bibitem{Ledrappier}
F.~Ledrappier.
\newblock On the dimension of some graphs.
\newblock In {\em Symbolic dynamics and its applications ({N}ew {H}aven, {CT},
  1991)}, volume 135 of {\em Contemp. Math.}, pages 285--293. Amer. Math. Soc.,
  Providence, RI, 1992.

\bibitem{SchiedCPPI}
A.~Schied.
\newblock Model-free {CPPI}.
\newblock {\em J. Econom. Dynam. Control}, 40:84--94, 2014.

\bibitem{SchiedStadje}
A.~Schied and M.~Stadje.
\newblock Robustness of delta hedging for path-dependent options in local
  volatility models.
\newblock {\em J. Appl. Probab.}, 44(4):865--879, 2007.

\bibitem{Sondermann}
D.~Sondermann.
\newblock {\em Introduction to stochastic calculus for finance}, volume 579 of
  {\em Lecture Notes in Economics and Mathematical Systems}.
\newblock Springer-Verlag, Berlin, 2006.

\bibitem{Takagi}
T.~Takagi.
\newblock A simple example of the continuous function without derivative.
\newblock In {\em Proc. Phys. Math. Soc. Japan}, volume~1, pages 176--177,
  1903.

\end{thebibliography}
\bibliographystyle{abbrv}

\end{document}